\documentclass{amsart}

\newtheorem{theorem}{Theorem}[section]
\newtheorem{lemma}[theorem]{Lemma}
\newtheorem{proposition}[theorem]{Proposition}
\theoremstyle{definition}
\newtheorem{definition}[theorem]{Definition}
\newtheorem{example}[theorem]{Example}

\theoremstyle{remark}

\usepackage{epsfig}

\reversemarginpar

\begin{document}

\newtheorem{conjecture}{Conjecture}[section]
\newtheorem{corollary}{Corollary}[section]
\newtheorem{Note}{Note}[section]

\theoremstyle{definition}
\newtheorem{rules}{Rule}[section]

\theoremstyle{remark}

%    Absolute value notation

\title{An extension of the method of brackets. Part 1.}

\author{Ivan Gonzalez}
\address{Instituto de F\'{i}sica y Astronom\'{i}a, Universidad de Valparaiso, Valparaiso, Chile}
\email{ivan.gonzalez@uv.cl}

\author{Karen Kohl}
\address{Department of Mathematics, University of  Southern Mississippi,  Long Beach, MS 39560}
\email{karen.kohl@usm.edu}

\author{Lin Jiu}
\address{Research Institute for Symbolic Computation, Johannes Kepler University,
Linz, Austria}
\email{ljiu@risc.uni-linz.ac.at}

\author{Victor H. Moll}
\address{Department of Mathematics,
Tulane University, New Orleans, LA 70118}
\email{vhm@math.tulane.edu}



\numberwithin{equation}{section}

\newcommand{\imagpart}{\mathop{\rm Im}\nolimits}
\newcommand{\realpart}{\mathop{\rm Re}\nolimits}
\newcommand{\no}{\noindent}
\newcommand{\ift}{\int_{0}^{\infty}}
\newcommand{\ione}{\int_{0}^{1}}
\newcommand{\eqf}{\stackrel{\bullet}{=}}
\newcommand{\op}[1]{\ensuremath{\operatorname{#1}}}
\newcommand{\pFq}[5]{\ensuremath{{}_{#1}F_{#2} \left( \genfrac{}{}{0pt}{}{#3}
{#4} \bigg| {#5} \right)}}

\begin{abstract}
The method of brackets is an efficient method for the evaluation of a large class of definite integrals on the 
half-line. It is based on a small collection of rules, some of which are heuristic. The extension discussed here 
is based on the concepts of null and  divergent series. These are formal representations of 
functions, whose coefficients $a_{n}$ have meromorphic  representations for $n \in \mathbb{C}$,  but might
vanish or blow up when $n \in \mathbb{N}$. These ideas are illustrated with the evaluation of a variety of 
entries from the classical table of integrals by Gradshteyn and Ryzhik.
\end{abstract}

%\begin{center}
%\textbf{DECEMBER 10, 10:22AM, New Orleans time}
%\end{center}

\maketitle

\section{Introduction} 
\label{sec-intro}
\setcounter{equation}{0}

The evaluation of definite integrals have a long history dating from the work of Eudoxus of Cnidus (408-355 BC) 
with the creation of the method of exhaustion.  The history of this problem is reported in 
\cite{kallio-1966a}.  A large variety of methods developed for the evaluations of integrals may be found in 
older Calculus textbooks, such as those by J.~Edwards \cite{edwardsj-1922a,edwardsj-1922b}. As the 
number of examples grew, they began to be collected in \textit{tables of integrals}. The table compiled by I.~S.~Gradshteyn and I.~M.~Ryzhik \cite{gradshteyn-2015a} is the most widely used one, 
now in its $8^{th}$-edition.

\medskip

The interest of the last author in this topic began with entry $3.248.5$ in \cite{gradshteyn-2000a}
\begin{equation}
I  =  \ift (1+x^2)^{-3/2} \left[ \varphi(x) + \sqrt{\varphi(x)}  
\right]^{-1/2} \, dx \label{i1}
\end{equation}
\noindent
where $\displaystyle{ \varphi(x)  =  1 + \tfrac{4}{3}x^{2}(1+x^{2})^{-2}.} $ The value $\pi/2 \sqrt{6}$ given in the table is 
incorrect, as a  direct numerical evaluation will confirm. Since an evaluation of the integral still elude us, the
 editors of the table found an ingenious temporary solution to
this problem: it does not 
appear in \cite{gradshteyn-2007a} nor in the latest edition \cite{gradshteyn-2015a}.  This motivated an effort to present 
proofs of all entries in Gradshteyn-Ryzhik. It began with \cite{moll-2007a} and has continued with several short papers.
These have appeared in \textit{Revista Scientia}, the latest one being \cite{amdeberhan-2016b}.

\medskip

The work presented here deals with the  \textit{method of brackets}. This is a new method for integration 
developed  in \cite{gonzalez-2007a,gonzalez-2008a,gonzalez-2009a} in the context of integrals 
arising from Feynman diagrams. It consists of a small number of rules that converts the integrand into a 
collection of series. These rules are reviewed in Section \ref{sec-method}, it is important to emphasize that 
 \texttt{most of these rules are still not rigorously justified and currently should be considered a collection of 
 heuristic rules}. 

\medskip 

The success of the method depends on the ability to give closed-form expressions for 
these series. Some of these heuristic rules are currently being placed on solid 
ground \cite{amdeberhan-2012b}. The reader will find in 
\cite{gonzalez-2014a,gonzalez-2010a,gonzalez-2010b} a large collection of examples 
that illustrate the power and flexibility of this method. 

\medskip

The operational rules are described in Section \ref{sec-method}. The method applies to functions that can be 
expanded in a formal power series 
\begin{equation}
f(x)=\sum_{n=0}^{\infty}a(n)x^{\alpha n+\beta-1},
\label{series-f}
\end{equation}
\noindent
where $\alpha, \, \beta \in \mathbb{C}$ and the coefficients $a(n)\in \mathbb{C}$. (The extra $-1$ in the 
exponent is for a convenient formulation of the operational rules). The adjective \textit{formal}
refers to the fact that the expansion is used to integrate over $[0, \infty)$, even though it might be 
valid only on a proper subset of the half-line. 

\begin{Note}
\label{note-required}
There is no precise description of the complete class of functions $f$ for which the method can be applied. At the 
moment, it is a working assumption, that the coefficients $a(n)$ in \eqref{series-f} are expressions that admit a
 unique meromorphic continuation to $n \in \mathbb{C}$. This is required, since the method involves the evaluation
  of $a(n)$ for $n$ not a natural number, hence an extension 
is needed. For example, the Bessel function 
\begin{equation}
I_{0}(x) = \sum_{n=0}^{\infty} \frac{1}{n!^{2}} \left( \frac{x}{2} \right)^{2n}
\label{bessel-i0}
\end{equation}
\noindent
has $\alpha = 2, \, \, \beta = 1$ and $\displaystyle{ a(n) = 1/2^{2n} n!^{2}}$ can be written as 
$\displaystyle{ a(n) = 1/2^{2n} \Gamma^{2}(n+1)}$ and now the evaluation, say at $n = \tfrac{1}{2}$, is possible. The 
same observation holds for the Bessel function
\begin{equation}
J_{0}(x) = \sum_{n=0}^{\infty} \frac{(-1)^{n}}{n!^{2}} \left( \frac{x}{2} \right)^{2n}.
\label{bessel-j0a}
\end{equation}
\end{Note}

The goal of the present work is to produce \textit{non-classical series representations} for  functions $f$, which do not 
have expansions like \eqref{series-f}. These representations are formally of the type \eqref{series-f} but  some of the
coefficients $a(n)$ might be null or divergent.   The examples show how to use these 
representations in conjunction 
with the method of brackets to evaluate definite integrals. The examples presented here come from the table 
\cite{gradshteyn-2015a}. This process is, up to now, completely heuristic. These non-classical series are classified 
according to the following types:

\medskip

\noindent
$1) \, $ \texttt{Totally (partially) divergent series}. Each term (some of the terms) in the series is a 
divergent value. For example, 
\begin{equation}
\sum_{n=0}^{\infty} \Gamma(-n) x^{n} \text{ and } \sum_{n=0}^{\infty} \frac{ \Gamma(n-3)}{n!} x^{n}.
\end{equation}

\smallskip

\noindent
$2) \, $ \texttt{Totally (partially) null series}. Each term (some of the terms) in the series vanishes. For example,
\begin{equation}
\sum_{n=0}^{\infty} \frac{1}{\Gamma(-n)} x^{n} \text{ and } \sum_{n=0}^{\infty} \frac{1}{\Gamma(3-n)} x^{n}.
\end{equation}
\noindent
This type includes series where all but finitely many terms vanish. These are polynomials in the corresponding 
variable.

\smallskip

\noindent
$3) \, $  \texttt{Formally divergent series}. This is a classical divergent series: the terms are finite but the sum of 
the series diverges. For example,
\begin{equation}
\sum_{n=0}^{\infty} \frac{ n! ^{2} }{(n+1) \, (2n)!}  5^{n}.
\end{equation}

In spite of the divergence of these series, they will be used in combination with the method of brackets to evaluate 
a variety of definite integrals. Examples of these type of series are given next. 

\smallskip

Some examples of functions that admit non-classical representations are given next.

\smallskip 

\noindent
$\bullet$ The \textit{exponential integral} with the partially divergent  series 
\begin{equation}
\text{Ei}(-x) = - \int_{1}^{\infty} t^{-1} e^{-xt} \, dt = \sum_{n=0}^{\infty} (-1)^{n} \frac{x^{n}}{n \, n!}.
\label{ei-null}
\end{equation}

\smallskip 

\noindent
$\bullet$ The \textit{Bessel $K_{0}$-function}
\begin{equation}
K_{0}(x) = \int_{0}^{\infty} \frac{ \cos x t \, dt}{(t^{2}+1)^{1/2}}
\end{equation}
\noindent
with totally null representation 
\begin{equation}
K_{0}(x) = \frac{1}{x} \sum_{n=0}^{\infty} (-1)^{n}  \frac{\Gamma^{2}(n + \tfrac{1}{2} )}{n! \, \Gamma(-n)} 
\left( \frac{4}{x^{2}} \right)^{n}
\label{k0-null}
\end{equation}
\noindent
and the totally divergent one 
\begin{equation}
K_{0}(x) = \frac{1}{2} \sum_{n=0}^{\infty} (-1)^{n} \frac{\Gamma(-n)}{n!} \left( \frac{x^{2}}{4} \right)^{n}.
\label{k0-divergent}
\end{equation}

Section \ref{sec-method} presents the rules of the method of brackets. Section \ref{sec-independence} shows that 
the bracket series associated to an integral is independent of the presentation of the integrand. The remaining
sections use the method of brackets and non-classical series to evaluate definite integrals. Section \ref{sec-expi}
contains the exponential integral $\text{Ei}(-x)$ in the integrand, Section \ref{sec-tricomi} has the
 Tricomi function $U(a,b;x)$ (as an 
example of the confluent hypergeometric function), Section \ref{sec-airy} is dedicated to integrals with the 
Airy function $\text{Ai}(x)$ and then Section \ref{sec-bessel-nu} has the Bessel function $K_{\nu}(x)$, with 
special emphasis on $K_{0}(x)$. Section \ref{sec-producing} gives examples of definite integral 
whose value contains the Bessel function  $K_{\nu}(x)$. The final section has a new approach to the evaluation 
of bracket series, based on a differential equation involving parameters. 

The examples presented in the current work have appeared in the literature, where the reader will find 
proofs of these formulas by classical methods.  One of the goals of this work is to illustrate the flexibility 
of the method of brackets to evaluated these integrals.

\section{The method of brackets}
\label{sec-method}
\setcounter{equation}{0}

The method of brackets evaluates integrals 
over the half line $[0, \, \infty)$. It is  based on a small number of rules reviewed in this section. 

\begin{definition}
For $a \in \mathbb{C}$, the symbol 
\begin{equation}
\langle a \rangle =  \ift x^{a-1} \, dx 
\end{equation}
is the {\em bracket} associated to the (divergent) integral on the right. The 
symbol 
\begin{equation}
\phi_{n} = \frac{(-1)^{n}}{\Gamma(n+1)}
\end{equation}
\noindent
is called the {\em indicator} associated to the index $n$. The notation
$\phi_{n_{1}n_{2}\cdots n_{r}}$, or simply 
$\phi_{12 \cdots r}$, denotes the product 
$\phi_{n_{1}} \phi_{n_{2}} \cdots \phi_{n_{r}}$. 
\end{definition}

\begin{Note}
The indicator $\phi_{n}$ will be used in the series expressions used in the method of brackets. For instance
\eqref{ei-null} is written as 
\begin{equation}
\text{Ei}(-x) =  \sum_{n} \phi_{n} \frac{x^{n}}{n}
\label{ei-null1}
\end{equation}
\noindent
and \eqref{k0-divergent} as 
\begin{equation}
K_{0}(x) = \frac{1}{2} \sum_{n} \phi_{n}  \Gamma(-n) \left( \frac{x^{2}}{4} \right)^{n}.
\label{k0-divergent1}
\end{equation}
\noindent
In the process of implementing the method of brackets, these series will be evaluated for $n \in \mathbb{C}$, not 
necessarily positive integers. Thus the notation for the indices does not include its range of values. 
\end{Note}

\medskip

\noindent
{\bf {\em Rules for the production of bracket series}}

The first part of the method is to associate to the integral  
\begin{equation}
I (f) = \int_{0}^{\infty} f(x) \, dx
\end{equation}
\noindent
 a bracket series. This is done following two rules:

\smallskip

\noindent
${\mathbf{Rule \, \, P_{1}}}$.  Assume $f$ has the expansion
\begin{equation}
f(x)=\sum_{n=0}^{\infty}\phi_{n} a(n)x^{\alpha n+\beta-1}.
\end{equation}
Then $I(f)$ is assigned the  \emph{bracket series }
\begin{equation}
I(f) =\sum_{n}\phi_{n}a(n)\left\langle \alpha n+\beta\right\rangle .
\end{equation}

\begin{Note}
The series including the indicator $\phi_{n}$ have indices \textit{without} limits, since its 
evaluation requires to take $n$ outside $\mathbb{N}$.
\end{Note}

\smallskip

\noindent
${\mathbf{Rule \, \, P_{2}}}$. 
For $\alpha \in \mathbb{C}$, the multinomial power 
$(u_{1} + u_{2} + \cdots + u_{r})^{\alpha}$ is assigned the 
$r$-dimension bracket series 
\begin{equation}
\sum_{n_{1},n_{2},\ldots, n_{r}}
\phi_{n_{1}\, n_{2} \,  \cdots n_{r}}
u_{1}^{n_{1}} \cdots u_{r}^{n_{r}} 
\frac{\langle -\alpha + n_{1} + \cdots + n_{r} \rangle}{\Gamma(-\alpha)}.
\end{equation}
\noindent
The integer $r$ is called \textit{the dimension of the bracket series}.

\medskip

\noindent
{\bf {\em Rules for the evaluation of a bracket series}}

The next set of rules associates a complex number to a bracket series. 

\smallskip

\noindent
${\mathbf{Rule \, \, E_{1}}}$. 
The one-dimensional bracket series is assigned the value 
\begin{equation}
\sum_{n} \phi_{n} a(n) \langle \alpha n + b \rangle = 
\frac{1}{|\alpha|} a(n^{*}) \Gamma(-n^{*}),
\end{equation}
\noindent
where $n^{*}$ is obtained from the vanishing of the bracket; that is, $n^{*}$ 
solves $an+b = 0$. 

\begin{Note}
The rule $E_{1}$ is a version of the \textit{Ramanujan's Master Theorem}.  This theorem requires an 
extension of the coefficients $a(n)$ from $n \in \mathbb{N}$ to $n \in \mathbb{C}$. The assumptions imposed on 
the function $f$ is precisely for the application of this result. A complete justification of this 
rule is provided in \cite{amdeberhan-2012b}. \textit{Making the  remaining rules rigorous is the subject of active research. }
\end{Note}

\smallskip

The next rule provides a value for multi-dimensional bracket series where 
 the number of sums is equal to the number of brackets. 

\smallskip

\noindent
${\mathbf{Rule \, \, E_{2}}}$. 
Assume the matrix
$B = (b_{ij})$ is non-singular, then the assignment is 
\begin{equation}
\sum_{n_{1},n_{2}, \cdots,n_{r}}  \phi_{n_{1} \cdots n_{r}} 
a(n_{1},\cdots,n_{r}) 
\langle b_{11}n_{1} + \cdots + b_{1r}n_{r} + c_{1} \rangle \cdots
\langle b_{r1}n_{1} + \cdots + b_{rr}n_{r} + c_{r} \rangle 
\nonumber
\end{equation}
\begin{equation}
=  \frac{1}{| \text{det}(B) |} a(n_{1}^{*}, \cdots n_{r}^{*}) 
\Gamma(-n_{1}^{*}) \cdots \Gamma(-n_{r}^{*}) 
\nonumber
\end{equation}
\noindent
where $\{ n_{i}^{*} \}$ 
is the (unique) solution of the linear system obtained from the vanishing of 
the brackets. There is no assignment if $B$ is singular. 

\smallskip

\noindent
${\mathbf{Rule \, \, E_{3}}}$. 
Each representation of an integral by a bracket series has
associated an {\em index of the representation} via 
\begin{equation}
\text{index } = \text{number of sums } - \text{ number of brackets}.
\end{equation}
\noindent
In the case of a multi-dimensional bracket series of positive index, then the system generated by the vanishing of 
the coefficients has a number of free parameters. The solution is 
obtained by computing all the contributions of maximal rank in the system by selecting these free 
parameters. Series expressed in the same variable (or argument)  are added. 

\begin{example}
A generic bracket series of index $1$ has the form 
\begin{equation}
\sum_{n_{1}, \, n_{2}} \phi_{n_{1},n_{2}} C(n_{1},n_{2}) A^{n_{1}} B^{n_{2}} 
\langle a_{11}n_{1} + a_{12}n_{2} + c_{1} \rangle,
\end{equation}
\noindent
where $a_{11}, \, a_{12}, \, c_{1}$ are fixed coefficients, $A, \, B$ are parameters and $C(n_{1},n_{2})$ is a function of the indices.

The Rule $E_{3}$ is used to generate two series by leaving first $n_{1}$ and then $n_{2}$ as free parameters. The Rule $E_{1}$ is 
used to assign a value to the corresponding series:

\smallskip

\noindent
$n_{1}$ as a free parameter produces 
\begin{equation*}
T_{1} = \frac{B^{-c_{1}/a_{12}}}{|a_{12}|} \sum_{n_{1}=0}^{\infty} \phi_{n_{1}} 
\Gamma \left( \frac{a_{11}n_{1}+c_{1}}{a_{12}} \right) 
C\left( n_{1}, - \frac{a_{11}n_{1} + c_{1}}{a_{12}} \right) \left(AB^{-a_{11}/a_{12}} \right)^{n_{1}};
\end{equation*}

\smallskip

\noindent
$n_{2}$ as a free parameter produces 
\begin{equation*}
T_{2} = \frac{A^{-c_{1}/a_{11}}}{|a_{11}|} \sum_{n_{2}=0}^{\infty} \phi_{n_{2}} 
\Gamma \left( \frac{a_{12}n_{2}+c_{1}}{a_{11}} \right) 
C\left( - \frac{a_{12}n_{2} + c_{1}}{a_{11}}, n_{2} \right) \left( BA^{-a_{12}/a_{11}} \right)^{n_{2}}.
\end{equation*}

\smallskip

The series $T_{1}$ and $T_{2}$ are expansions of the solution in terms of different parameters 
\begin{equation}
x_{1}  = AB^{-a_{11}/a_{12}} \,\,\, {\rm{ and }} \,\,\, x_{2} = BA^{-a_{12}/a_{11}}. 
\end{equation}
\noindent
Observe that $x_{2} = x_{1}^{a_{12}/a_{11}}$. Therefore 
the bracket series is assigned the value $T_{1}$ \textit{or} $T_{2}$. If one of the series is a null-series or divergent, it is 
discarded.  If \textit{both} series are discarded, the method of brackets does not produce a value for the integral that generates the 
bracket series.

\smallskip

Some special cases will clarify the rules to follow in the use of the series $T_{1}$ and $T_{2}$. Suppose $a_{12} = -a_{11}$, then 
\begin{equation}
T_{1} = \frac{B^{-c_{1}/a_{11}}}{|a_{11}|} 
\sum_{n_{1}=0}^{\infty} \phi_{n_{1}} \Gamma \left(  n_{1} + \frac{c_{1}}{a_{11}} \right) 
C \left( n_{1}, -n_{1} - \frac{c_{1}}{a_{11}} \right) (AB)^{n_{1}}
\end{equation}
\noindent
and 
\begin{equation}
T_{2} = \frac{A^{-c_{1}/a_{11}}}{|a_{11}|} 
\sum_{n_{2}=0}^{\infty} \phi_{n_{2}} \Gamma \left(  n_{2} + \frac{c_{1}}{a_{11}} \right) 
C \left( -n_{2} -  \frac{c_{1}}{a_{11}}, n_{2} \right) (AB)^{n_{2}}
\end{equation}
\noindent
and since both series are expansions in the same parameter $( AB )$,  \textit{their values must be added}  to compute the 
value associated to the bracket series.  On the other hand, if $a_{12} = -2a_{11}$, then 
\begin{equation*}
T_{1} = \frac{B^{c_{1}/2a_{11}}}{2 |a_{11}|} \sum_{n_{1}=0}^{\infty} \phi_{n_{1}} 
\Gamma \left( - \frac{1}{2} n_{1} - \frac{c_{1}}{2a_{11}} \right) 
C \left( n_{1},  \frac{1}{2} n_{1} + \frac{c_{1}}{2a_{11}} \right) 
\left( A B^{1/2} \right)^{n_{1}}
\end{equation*}
\noindent
and 
\begin{equation*}
T_{2} = \frac{A^{-c_{1}/a_{11}}}{ |a_{11}|} \sum_{n_{2}=0}^{\infty} \phi_{n_{2}} 
\Gamma \left( -2  n_{2} + \frac{c_{1}}{a_{11}} \right) 
C \left( 2 n_{2} - \frac{c_{1}}{a_{11}}, n_{2} \right) 
\left( A^{2} B \right)^{n_{2}}.
\end{equation*}
\noindent 
Splitting the sum in $T_{1}$ according to the parity of the  indices produces a power series in $A^{2}B$ when 
$n_{1} = 2 n_{3}$ is  even and for $n_{1}$ odd a second power series in the same argument $A^{2}B$ times an extra factor 
$AB^{1/2}$. Since these are expansions in the same argument, they have to be added to count their contribution to the 
bracket series. 
\end{example}

\smallskip

\begin{Note}
It is important to observe that the index is attached to a specific 
representation of the integral and not just to integral itself.  The 
experience obtained by the authors using this method suggests that, among 
all representations of an integral as a bracket series, the one with 
{\em minimal index} should be chosen. 
\end{Note}

\begin{Note}
 The extension presented in this work shows how 
to use these divergent series in the evaluation of definite integrals. Example \ref{example-6-222} illustrates 
this procedure.
\end{Note}

\smallskip

\noindent
${\mathbf{Rule \,\, E_{4}}}$. In the evaluation of a bracket series, repeated series are counted only once. For instance, a 
convergent series appearing repeated  in the 
same region of convergence should be counted only once. The same treatment should be given to null and 
divergent series.

\begin{Note}
Example \ref{ex-rule-e4}  in Section \ref{sec-expi} illustrates the use of this rule.
\end{Note}

\begin{Note}
A systematic procedure in the simplification of the series  has been used throughout the 
literature: express factorials in terms of the gamma function and the transform quotients of gamma terms into 
Pochhammer symbols,  defined by
\begin{equation}
(a)_{k} = a(a+1) \cdots (a+k-1) = \frac{\Gamma(a+k)}{\Gamma(a)}.
\label{gamma-poch}
\end{equation}
\noindent
Any presence of a Pochhammer with a negative index $k$ is transformed by the rule 
\begin{equation}
(a)_{-k} = \frac{(-1)^{k}}{(1-a)_{k}}, \quad \text{ for } k \in \mathbb{N}.
\label{rule-11}
\end{equation}
\noindent
In the special case when $a$  \textit{is also} a negative integer, the rule 
\begin{equation}
(-km)_{-m} = \frac{k}{k+1} \cdot \frac{(-1)^{m} (km)!}{((k+1)m)!}
\end{equation}
\noindent
holds. This value is justified in \cite{gonzalez-2016a}. The duplication formula
\begin{equation}
(a)_{2n} = 2^{2n} \left( \frac{a}{2} \right)_{n} \left( \frac{a+1}{2} \right)_{n}
\label{poch-dupl}
\end{equation}
\noindent
is also used in the simplifications.

Many of the evaluations are given as values of the hypergeometric functions 
\begin{equation}
 _{p}F_{q}\left(\genfrac{}{}{0pt}{}{a_{1},\ldots,a_{p}}{b_{1},\ldots,b_{q}}\bigg{|}z\right) = 
\sum_{n=0}^{\infty} \frac{(a_{1})_{n} \cdots (a_{p})_{n}}{(b_{1})_{n} \cdots (b_{q})_{n} } \frac{z^{n}}{n!},
\label{hyper-def}
\end{equation}
\noindent
with $(a)_{n}$ as in \eqref{gamma-poch}. It is often that the value of $_{2}F_{1}$ at $z=1$ is required. This is given 
by the classical formula of Gauss:
\begin{equation}
 \label{gauss-value}
 \pFq21{a  \,\,\, b}{c}{1}  = \frac{\Gamma(c) \Gamma(c-a-b)}{\Gamma(c-a) \, \Gamma(c-b)}.
 \end{equation}

\end{Note}

\begin{Note}
The extension considered here is to use  the method of brackets to functions that do not admit a series 
representation as described in Rule $P_{1}$.  For example, the Bessel function $K_{0}(x)$ has a singular expansion of 
the form 
\begin{equation}
\label{exp-k0}
K_{0}(x) = - \left( \gamma - \ln 2 +  \ln x \right) I_{0}(x) + 
\sum_{j=0}^{\infty} \frac{H_{j}}{j!^{2}}  \frac{x^{2j}}{2^{2j}} 
\end{equation}
\noindent
(see  \cite[10.31.2]{olver-2010a}). Here 
$I_{0}(x)$ is the Bessel 
function given in \eqref{bessel-i0},
$\begin{displaystyle}H_{j} = \sum_{k=1}^{j} \frac{1}{k}\end{displaystyle}$ is the harmonic 
number and  $\gamma = \lim\limits_{j \to \infty}  \left( H_{j} - \ln j \right)$ is Euler's constant. The presence of the logarithm 
term in \eqref{exp-k0} does not permit a direct application of the 
method of brackets. An alternative is presented in  Section \ref{sec-bessel-nu}.
\end{Note}

\section{Independence of the factorization}
\label{sec-independence}
\setcounter{equation}{0}

The evaluation of a definite integral by the method of brackets begins with the association of a bracket series to 
the integral.  It is common that the integrand contains several factors from which the bracket series 
is generated. This representation is not unique. For example, the integral 
\begin{equation}
\label{int-j01}
I = \int_{0}^{\infty} e^{-ax} J_{0}(x) \, dx 
\end{equation}
\noindent
is associated the bracket series 
\begin{equation}
\sum_{n_{1},n_{2}} \phi_{n_{1},n_{2}} \frac{a^{n_{1}}}{2^{2n_{2}} \Gamma(n_{2}+1)}
\langle n_{1} + 2n_{2} + 1 \rangle,
\end{equation}
\noindent 
and rewriting \eqref{int-j01} as 
\begin{equation}
I = \int_{0}^{\infty} e^{-ax/2} e^{-ax/2} J_{0}(x) \, dx,
\end{equation}
\noindent
provides the second  bracket series 
\begin{equation}
\sum_{n_{1},n_{2},n_{3}} \phi_{n_{1},n_{2},n_{3}} 
\frac{a^{n_{1}+n_{2}}}{2^{n_{1}+n_{2}+2n_{3}} \Gamma(n_{3}+1)} 
\langle n_{1} + n_{2} + 2n_{3} + 1 \rangle
\end{equation}
\noindent
associated to \eqref{int-j01}. It is shown next that all such bracket series representations of an integral 
produce the same the value.

\begin{theorem}
\label{thm:TwoFactors}
Assume  $f(x)=g(x)h(x)$, where  $f, \, g \text{ and }h$ have expansions as in \eqref{series-f}.  Then, the 
method of brackets assigns  the same value to the  integrals
\begin{equation}
I_{1}=\int_{0}^{\infty}f(x) \, dx\text{ and }I_{2}=\int_{0}^{\infty}g(x)h(x) \, dx.
\end{equation}
\end{theorem}
\begin{proof}
Suppose that 
\begin{eqnarray*}
f(x) & = &\displaystyle{ \sum_{n} \phi_{n}a(n)x^{\alpha n+\beta} } \label{expan-f}\\
g(x)& = &\displaystyle{ \sum_{n_{1}} \phi_{n_{1}}b\left(n_{1}\right)x^{\alpha n_{1}+\beta_{1}}} \\
h(x)& = &\displaystyle{ \sum_{n_{2}}\phi_{n_{2}}c\left(n_{2}\right)x^{\alpha n_{2}+\beta_{2}} }.
\end{eqnarray*}
Then
\begin{equation}
\label{value-I1}
I_{1}=\int_{0}^{\infty}f(x) dx=\sum_{n}\phi_{n}a\left(n\right)\left\langle \alpha n+\beta + 1 \right\rangle =
\frac{1}{|\alpha|} a(-s)\Gamma(s),
\end{equation}
\noindent
with $s = (1+\beta)/\alpha$.

\medskip

To evaluate the second integral, observe that 
\begin{eqnarray*}
g(x)h(x) & = & x^{\beta_{1}+\beta_{2}} \left( \sum_{n_{1}=0}^{\infty} \frac{(-1)^{n_{1}}}{n_{1}!}
b(n_{1}) x^{\alpha n_{1}} \right)
 \left( \sum_{n_{2}=0}^{\infty} \frac{(-1)^{n_{2}}}{n_{2}!} c(n_{2}) x^{\alpha n_{2}} \right) \\
 & = & x^{\beta_{1}+\beta_{2}} \sum_{n=0}^{\infty} F(n) x^{\alpha n},  \nonumber 
 \end{eqnarray*}
 \noindent
 with 
 \begin{eqnarray}
 F(n) & = & \sum_{k=0}^{n} \frac{(-1)^{k}}{k!}  b(k) \frac{(-1)^{n-k}}{(n-k)!} c(n-k) \\
 & = & \frac{(-1)^{n}}{n!} \sum_{k=0}^{n} \binom{n}{k} b(k) c(n-k). \nonumber 
 \end{eqnarray}
 \noindent
 This yields 
 \begin{equation}
 f(x) = \sum_{n=0}^{\infty} \frac{(-1)^{n}}{n!}\left[ \sum_{k=0}^{n} \binom{n}{k} b(k) c(n-k) \right] x^{\alpha n + \beta_{1} + 
 \beta_{2}} 
 \end{equation}
 \noindent
 and matching this with \eqref{expan-f} gives $\beta = \beta_{1}+\beta_{2}$ and 
 \begin{equation}
 a(n)   =  \sum_{k=0}^{n} \binom{n}{k} b(k) c(n-k) 
    =  \sum_{k=0}^{\infty} \frac{(-1)^{k}}{k!} \,  (-n)_{k} b(k)c(n-k). \label{identity-2}
   \end{equation}
   
   Now, the method of brackets gives 
   \begin{equation}
  I_{2} =  \int_{0}^{\infty} g(x)h(x) \, dx = \sum_{n_{1},n_{2}} \phi_{n_{1},n_{2}} b(n_{1}) c(n_{2}) 
   \langle \alpha n_{1} + \alpha n_{2} + \beta + 1 \rangle
   \end{equation}
   \noindent
   and it yields two series as solutions
   \begin{eqnarray}
   T_{1} & = & \frac{1}{| \alpha |} \sum_{n} \phi_{n} \Gamma \left( n + s  \right) 
   b(n) c(-n - s) \\
    T_{2} & = & \frac{1}{| \alpha |} \sum_{n} \phi_{n} \Gamma \left( n + s  \right) 
   b(-n-s) c(n), \nonumber 
   \end{eqnarray}
   \noindent
   with $s = (\beta+1)/\alpha$. Comparing with \eqref{value-I1} shows that $I_{1} = I_{2}$ is equivalent to 
   \begin{equation}
   \label{required-1}
   \Gamma(s) a( - s) = 
    \sum_{n} \phi_{n} \Gamma(n+s) b(n) c(-s-n),
    \end{equation}
    \noindent
    that is, 
    \begin{equation}
    a(-s) = \sum_{n} \phi_{n} (s)_{n} b(n) c(-s-n).
    \label{required-1}
    \end{equation}
    \noindent
    The identity \eqref{required-1} is the extension of \eqref{identity-2} from $n \in \mathbb{N}$ to $s \in \mathbb{C}$.
    This extension is part of the requirements on the functions $f$ explained in Note \ref{note-required}. The proof 
    is complete.
\end{proof}
\smallskip{}

It is direct to extend the result to the case of a finite number of factors.

\begin{theorem}
\label{thm:Independence} Assume $f$ admits a representation of the
form $f\left(x\right)=\overset{r}{\underset{i=1}{\prod}}f_{i}\left(x\right)$.
Then the value of the integral, obtained by method of brackets, is
the same for both series representations.
\end{theorem}

\section{The exponential integral}
\label{sec-expi}
\setcounter{equation}{0}

The \textit{exponential integral function} is defined by the integral formula
\begin{equation}
\text{Ei}(-x) = - \int_{1}^{\infty} \frac{\exp(- x t )}{t} \, dt, \text{ for } x > 0.
\end{equation}
(See \cite[$8.211.1$]{gradshteyn-2015a}).  The method of brackets is now used to produce a non-classical series 
for this function. Start by  replacing 
the exponential function by its power series to obtain
\begin{equation}
\label{ei-series1}
\text{Ei}(-x)  =   - \sum_{n_{1}} \phi_{n_{1}} x^{n_{1}} \int_{1}^{\infty} t^{n_{1}-1} \, dt 
\end{equation}
\noindent
 and then use  the method of brackets  to produce 
 \begin{equation*}
 \int_{1}^{\infty} t^{n_{1}-1} \, dt = \int_{0}^{\infty} (y+1)^{n_{1}-1} \, dy = 
 \sum_{n_{2}, \, n_{3}} \phi_{n_{2}n_{3}} \frac{ \langle - n_{1} + 1 + n_{2} + n_{3} \rangle \, \langle n_{2} +1 \rangle }
 {\Gamma(-n_{1}+1)}.
 \end{equation*}
 \noindent
 Replace this  in \eqref{ei-series1} to obtain 
 \begin{equation}
 \text{Ei}(-x) = - \sum_{n_{1},n_{2},n_{3}} \phi_{n_{1}n_{2}n_{3}} x^{n_{1}} 
 \frac{\langle -n_{1}+1+n_{2}+n_{3} \rangle \,\, \langle n_{2} +1 \rangle }{\Gamma(-n_{1}+1)}.
 \end{equation}
 
 The evaluation of this series by the method of brackets generates two identical terms for $\text{Ei}(-x)$: 
 \begin{equation}
 \text{Ei}(-x) = \sum_{n=0}^{\infty} \frac{(-1)^{n}}{n \, \Gamma(n+1)} x^{n}.
 \label{divergent-ei}
 \end{equation}
 \noindent
 Only one of them is kept, according  to Rule $E_{4}$. This is a partially divergent series (from the value at $n=0$), 
 written as 
 \begin{equation}
 \text{Ei}(-x) = \sum_{n} \phi_{n} \frac{x^{n}}{n}.
 \label{pds-ei1}
 \end{equation}
 
 The next example illustrates how to use this partially divergent series in the evaluation of an integral.
 
 \begin{example}
 Entry $6.223$ of \cite{gradshteyn-2015a} gives the Mellin transform of the exponential integral as 
 \begin{equation}
 \int_{0}^{\infty} x^{\mu-1} \text{Ei}(- b x) \, dx = - \frac{b^{-\mu}}{\mu} \Gamma(\mu).
 \end{equation}
 \noindent
 To verify this, use the partially divergent series \eqref{pds-ei1}  and  the method of brackets to obtain 
 \begin{eqnarray}
 \int_{0}^{\infty} x^{\mu-1} \text{Ei}(-bx) \, dx & = & \sum_{n} \phi_{n} \frac{b^{n}}{n} \int_{0}^{\infty} x^{\mu+n-1} \, dx \\
 & = & \sum_{n} \phi_{n} \frac{b^{n}}{n} \langle \mu + n \rangle \nonumber  \\
 & = & - \frac{b^{-\mu}}{\mu} \Gamma(\mu), \nonumber 
 \end{eqnarray}
 \noindent
 as claimed.
 \end{example}

 \begin{example}
 Entry $6.228.2$ in \cite{gradshteyn-2015a} is 
 \begin{equation}
 \label{formG}
G(\nu,\mu,\beta) =  \int_{0}^{\infty} x^{\nu-1} e^{-\mu x} \text{Ei}(-\beta x) \, dx = 
 - \frac{\Gamma(\nu)}{\nu (\beta + \nu)^{\nu}} \pFq21{1 \,\, \,\, \nu}{\nu+1}{\frac{\mu}{\beta+\mu}}.
 \end{equation}
 \noindent
 The partially divergent series \eqref{pds-ei1} is now used to establish this formula. First form 
 the bracket series 
 \begin{equation}
 G(\nu, \mu, \beta) = \sum_{n_{1},n_{2}} \phi_{n_{1},n_{2}} 
 \frac{\beta^{n_{1}} \mu^{n_{2}}}{n_{1}} \langle n_{1} + n_{2} + \nu \rangle.
 \end{equation}
 \noindent
 Rule $E_{1}$ yields two cases from  the equation $n_{1}+n_{2}+\nu=0$:
 
 \smallskip
 
 \noindent
 \textit{Case 1}: $n_{2} = - n_{1} - \nu$ produces 
 \begin{equation}
 T_{1} = \mu^{-\nu} \sum_{n_{1}=0}^{\infty} \frac{(-1)^{n_{1}}}{n_{1}!} \frac{\Gamma(n_{1}+\nu)}{n_{1}}
 \left(  \frac{\beta}{\mu} \right)^{n_{1}},
 \end{equation}
 \noindent
 which is discarded since it is partially divergent (due to  the term $n_{1}=0$).
 
 \smallskip
 
 \noindent
 \textit{Case 2}: $n_{1} = -n_{2}-\nu$ gives 
 \begin{equation}
 \label{4-11}
 T_{2} = - \beta^{-\nu} \sum_{n_{2}=0}^{\infty} \frac{(-1)^{n_{2}}}{n_{2}!} 
 \left( \frac{\mu}{\beta} \right)^{n_{2}} \frac{\Gamma(n_{2}+\nu)}{n_{2}+\nu},
 \end{equation}
 \noindent
 and using 
 \begin{equation}
 \Gamma(n_{2}+\nu) = (\nu)_{n_{2}} \Gamma(\nu) 
 \text{ and } n_{2}+\nu = \frac{\Gamma(n_{2}+\nu+1)}{\Gamma(n_{2}+\nu)} =
 \frac{(\nu+1)_{n_{2}} \Gamma(\nu+1)}{(\nu)_{n_{2}} \Gamma(\nu)}
 \end{equation}
 \noindent
 equation \eqref{4-11} becomes 
 \begin{eqnarray}
 T_{2} & = & - \frac{\Gamma(\nu)}{\nu \, \beta^{\nu}} \sum_{n_{2}=0}^{\infty} 
 \frac{(\nu)_{n_{2}} (\nu)_{n_{2}}}{n_{2}! \, (\nu+1)_{n_{2}}} 
 \left( - \frac{\mu}{\beta} \right)^{n_{2}} \\
 & = & - \frac{\Gamma(\nu)}{\nu \, \beta^{\nu}} 
 \pFq21{\nu \,\,\,\,\, \nu}{\nu+1}{ - \frac{\mu}{\beta}}. \nonumber 
  \end{eqnarray}
  \noindent
  The condition $|\mu| < |\beta|$ is imposed to guarantee the convergence of the series.  Finally,  the 
  transformation rule (see entry $9.131.1$ in \cite{gradshteyn-2015a})
  \begin{equation}
  \pFq21{\alpha \,\,\,\, \beta}{\gamma}{z} = (1-z)^{-\alpha} \pFq21{\alpha \,\,\,\, \gamma - \beta}{\gamma}
  {\frac{z}{z-1}}
  \end{equation}
  \noindent
  with $\alpha = \beta = \nu, \, \gamma = \nu+1$ and $z = - \mu/\beta$ yields  \eqref{formG}.
 \end{example}
 
 \begin{example}
 The next evaluation is entry $6.232.2$ in \cite{gradshteyn-2015a}:
 \begin{equation}
 G(a,b) = \int_{0}^{\infty} \text{Ei}(- a x) \cos bx \, dx = - \frac{1}{b} \tan^{-1} \left( \frac{b}{a} \right).
 \label{G-form1}
 \end{equation}
 \noindent
 A direct application of the method of brackets using 
 \begin{equation}
\cos x = \pFq01{-}{\tfrac{1}{2}}{- \frac{x^{2}}{4}}
\label{cosine-hyper}
\end{equation}
  gives 
 \begin{equation}
 G(a,b) = \sqrt{\pi} \sum_{n_{1}, n_{2}} \phi_{n_{1},n_{2}} 
 \frac{b^{2n_{1}}a^{n_{2}}}{2^{2n_{1}} \Gamma( n_{1} + \tfrac{1}{2}) \, n_{2}} \langle 2n_{1} + n_{2} + 1 \rangle.
 \end{equation}
 \noindent
 This produces two series for $G(a,b)$: 
 \begin{equation}
 T_{1} = \frac{\sqrt{\pi}}{b} \sum_{n_{2}=0}^{\infty} \frac{(-1)^{n_{2}}}{n_{2}!} \frac{\Gamma(\tfrac{1}{2}(n_{2}+1))}
 {n_{2} \,  \Gamma( - \frac{1}{2}n_{2})} \left( \frac{2a}{b} \right)^{n_{2}},
 \label{form-T1}
 \end{equation}
 \noindent
 and 
 \begin{equation}
 T_{2} = - \frac{\sqrt{\pi}}{a} 
 \sum_{n_{1}=0}^{\infty} \frac{(-1)^{n_{1}}}{n_{1}!}  \frac{\Gamma(2n_{1}+1)}{(2n_{1}+1) \Gamma(n_{1} + \tfrac{1}{2}) }
 \left( \frac{b^{2}}{4a^{2}} \right)^{n_{1}}.
 \end{equation}
 \noindent
 The analysis begins with a simplification of $T_{2}$. Use the duplication formula for the gamma function 
 \begin{equation}
 \frac{\Gamma(2u)}{\Gamma(u)}  = \frac{2^{2u-1}}{\sqrt{\pi}} \Gamma(u + \tfrac{1}{2})
 \end{equation}
 \noindent
 and write 
 \begin{equation}
 \frac{1}{2n_{1}+1 }= \frac{(1)_{n_{1}} \left( \tfrac{1}{2} \right)_{n_{1}}}{n_{1}! \, \left( \tfrac{3}{2} \right)_{n_{1}}}
 \end{equation}
 \noindent
 to obtain 
 \begin{equation}
 T_{2} = - \frac{1}{a} \pFq21{1 \,\,\, \tfrac{1}{2}}{\tfrac{3}{2}}{ - \frac{b^{2}}{a^{2}}},
 \end{equation}
 \noindent
 provided $|b|<|a|$ to guarantee convergence. The form \eqref{G-form1} comes from the identity 
 \begin{equation}
 \pFq21{ \tfrac{1}{2} \,\,\, 1}{\tfrac{3}{2}}{-z^{2}} = \frac{\tan^{-1} z}{z}
 \end{equation}
 (see $9.121.27$ in \cite{gradshteyn-2015a}).  
 
 \smallskip 
 
 The next step is the evaluation of $T_{1}$. Separating the sum \eqref{form-T1} into even and odd indices yields 
 \begin{eqnarray}
 T_{1} & = & \frac{\sqrt{\pi}}{2 b} \sum_{n=0}^{\infty} \frac{1}{(2n)!} \frac{\Gamma \left( n + \tfrac{1}{2} \right) }{n \Gamma(-n)} 
 \left( \frac{4a^{2}}{b^{2}} \right)^{n} \\
 & & - \frac{\sqrt{\pi}}{b} \sum_{n=0}^{\infty} \frac{1}{(2n+1)!} 
 \frac{\Gamma(n+1)}{(2n+1) \Gamma \left( -n - \tfrac{1}{2} \right)} \left( \frac{2a}{b} \right)^{2n+1},
 \nonumber 
 \end{eqnarray}
 \noindent
 and in hypergeometric form
 \begin{eqnarray}
 T_{1} & = &  - \frac{\pi}{2b} \,\, \pFq21{0 \,\,\, \tfrac{1}{2}}{\tfrac{1}{2}}{- \frac{a^{2}}{b^{2}}} + 
 \frac{a}{b^{2}} \,\, \pFq21{\tfrac{1}{2} \,\,\,  1 }{\tfrac{3}{2}}{ - \frac{a^{2}}{b^{2}}} \\
 & = & - \frac{\pi}{2b} + \frac{1}{b} \tan^{-1} \left( \frac{a}{b} \right). \nonumber 
 \end{eqnarray}
 \noindent
 and this is the same as \eqref{G-form1}.
 
 \smallskip
 
 The evaluation of entry $6.232.1$ in \cite{gradshteyn-2015a}
 \begin{equation}
 \int_{0}^{\infty} \text{Ei}(- a x) \sin bx \, dx = - \frac{1}{2b} \ln \left( 1 + \frac{b^{2}}{a^{2}} \right)
 \end{equation}
 \noindent
 is obtained in a similar form.
 \end{example}
 
 \begin{example}
 Entry $6.782.1$  in \cite{gradshteyn-2015a} is
 \begin{equation}
 B(z) =  \int_{0}^{\infty} \text{Ei}(-x) J_{0}(2 \sqrt{zx}) \, dx = \frac{e^{-z}-1}{z}.
 \end{equation}
 \noindent
 Here 
 \begin{equation}
 J_{0}(x) = \pFq01{-}{1}{-\frac{x^{2}}{4}} 
 \end{equation}
 \noindent
 is the classical Bessel function defined in \eqref{bessel-j0a}. Therefore 
 \begin{equation}
 J_{0}(2 \sqrt{z x }) = \sum_{n_{2}} \phi_{n_{2}} \frac{z^{n_{2}}}{\Gamma(n_{2}+1)} x^{n_{2}}.
 \end{equation}
 \noindent
 The standard procedure using the partially divergent series \eqref{divergent-ei} now gives 
 \begin{equation}
 B(z) = \sum_{n_{1},n_{2}} \phi_{n_{1},n_{2}} 
 \frac{1}{n_{1}} \frac{z^{n_{2}}}{\Gamma(n_{2}+1)} \langle n_{1}+n_{2}+1 \rangle,
 \end{equation}
 \noindent
 which gives the convergent series 
  \begin{equation}
 T_{1}  =   - \sum_{n_{1}=0}^{\infty} \frac{(-1)^{n_{1}}}{n_{1}!}  \frac{(1)_{n_{1}}}{(2)_{n_{1}}} z^{n_{1}}  
  =  - \pFq11{1}{2}{-z} 
  =  \frac{e^{-z}-1}{z},
 \end{equation}
 \noindent
 and the series
 \begin{equation}
 T_{2} = - \frac{1}{z} \sum_{n_{2}=0}^{\infty}  \frac{(-z)^{-n_{2}}}{\Gamma(1- n_{2})}.
 \end{equation}
 \noindent
 Observe that the expression $T_{2}$ contains a single non-vanishing term, so it is of 
 the partially null type. An alternative form of $T_{2}$ is 
 to write 
 \begin{eqnarray}
 T_{2} & = & - \frac{1}{z} \sum_{n_{2}=0}^{\infty} \frac{(-z^{-1})^{n_{2}}}{\Gamma(1-n_{2})} \\
 & = & - \frac{1}{z} \sum_{n_{2}=0}^{\infty} \frac{(-z^{-1})^{n_{2}} }{\Gamma(1) \, (1)_{-n_{2}}} \nonumber \\
  & = & -\frac{1}{z} \sum_{n_{2}=0}^{\infty} (z^{-1})^{n_{2}} (0)_{n_{2}} \,\, (1)_{n_{2} }\,\, \frac{(z^{-1})^{n_{2}}}{n_{2}!} 
  \nonumber \\
  & = & - \frac{1}{z} \pFq20{0 \,\,\, 1}{-}{\frac{1}{z}}. \nonumber 
 \end{eqnarray}
 \noindent
 The series $\begin{displaystyle}\pFq20{a \,\, b}{-}{z} \end{displaystyle}$ diverges, unless one of the 
 parameters $a$ or $b$ 
 is a non-positive integer, in which case the series terminates and it reduces to a polynomial. This is precisely what 
 happens here: only the term for $n_{2}=0$ is non-vanishing and $T_{2}$ reduces to 
 \begin{equation}
 T_{2} = - \frac{1}{z}.
 \end{equation}
 \noindent
 This  gives the asymptotic behavior $B(z) \sim - 1/z$, consistent with the value of $T_1$ for large $z$. This phenomena occurs 
 every time one obtains a series of the form ${_{p}F_{q}}(z)$ with $p \geq q+2$ when the series diverges. The truncation represents 
 an asymptotic approximation of the solution.
 \end{example}
 
 \section{The Tricomi function}
\label{sec-tricomi}
\setcounter{equation}{0}

The confluent hypergeometric function, denoted by  $\pFq11{a}{c}{z}$,  defined in \eqref{hyper-def}, arises when
 two of the regular singular points of the differential 
equation for the
Gauss hypergeometric function $\pFq21{a \,\, b}{c}{z}$, given by
\begin{equation}
z(1-z)y''+(c-(a+b+1)z)y'-aby=0,
\end{equation}
are allowed to merge into one singular point. More specifically, if we replace $z$ by $z/b$ in
$\pFq21{a \,\, b }{c}{z}$, then the corresponding differential equation has singular points at $0$, $b$ and
$\infty$. Now let $b\to\infty$ so as to have infinity as a confluence of two singularities. This results in the
function $\pFq11{a}{c}{z}$ so that
\begin{equation}
\pFq11{a}{c}{z} =\lim_{b\to\infty}   \pFq21{a \,\,\, b}{c}{\frac{z}{b}},
\end{equation}
and the corresponding differential equation
\begin{equation}\label{che}
zy''+(c-z)y'-ay=0,
\end{equation}
known as the confluent hypergeometric equation.  Evaluation of integrals connected to this equation 
are provided in \cite{dixit-2015f}. 

 The equation \eqref{che} has two linearly independent solutions:
\begin{equation}
M(a,b;x) = \pFq11{a}{b}{x},
\end{equation}
\noindent
known as the Kummer function and the \textit{Tricomi function} with integral representation 
\begin{equation}
U(a,b;x) = \frac{1}{\Gamma(a)} \int_{0}^{\infty} t^{a-1} \exp(-xt) (1+t)^{b-a-1} \, dt,
\end{equation}
\noindent
and hypergeometric form 
\begin{equation}
\label{hyper-tricomi}
U(a,b;x) = \frac{\Gamma(b-1)}{\Gamma(a)} x^{1-b} 
\pFq11{1+a-b}{2-b}{x} + 
\frac{\Gamma(1-b)}{\Gamma(1+a-b)} 
\pFq11{a}{b}{x}.
\end{equation}

\medskip

A direct application of the method of brackets gives 
\begin{eqnarray*}
U(a,b;x) & = &  \frac{1}{\Gamma(a)} \int_{0}^{\infty} t^{a-1} 
\left( \sum_{n_{1}} \phi_{n_{1}} x^{n_{1}} t^{n_{1}} \right) 
\left( \sum_{n_{2},n_{3}} \phi_{n_{2},n_{3}} t^{n_{3}}
 \frac{ \langle 1+ a - b + n_{2} + n_{3} \rangle }{\Gamma(1+a-b)} \right) \, dt \\
& = & 
\frac{1}{\Gamma(a)} \sum_{n_{1},n_{2},n_{3}} \phi_{n_{1},n_{2},n_{3}} x^{n_{1}} 
\frac{ \langle 1 + a -b +n_{2} + n_{3} \rangle }{\Gamma(1+a-b)} \langle a + n_{1} + n_{3} \rangle.
\end{eqnarray*}

This is a bracket series of index $1$ and its evaluation produces three terms:
\begin{eqnarray*}
U_{1}(a,b;x) & = &  \frac{\Gamma(1-b)}{\Gamma(1+a-b)} \pFq11{a}{b}{x},  \\
U_{2}(a,b;x) & = & \frac{\Gamma(b-1)}{\Gamma(a)} x^{1-b} \pFq11{1+a-b}{2-b}{x}, \\
U_{3}(a,b;x) & = & x^{-a} \,  \pFq20{a \,\,\,\, 1+a-b}{-}{- \frac{1}{x}}. 
\end{eqnarray*}

The first two are convergent in the region $|x|<1$ and their sum yields 
\eqref{hyper-tricomi}.  The series $U_{3}$ is formally divergent, the terms are finite but the series is divergent. 

\begin{example}
The Mellin transform of the Tricomi function is given by 
\begin{equation}
I(a,b;\beta) = \int_{0}^{\infty} x^{\beta-1} U(a,b,x) \, dx.
\label{mellin-tricomi}
\end{equation}
\noindent
Entry $7.612.1$ of \cite{gradshteyn-2015a}
\begin{equation}
\label{hyper-11}
\int_{0}^{\infty} x^{\beta-1} \pFq11{a}{b}{-x} \, dx = 
\frac{\Gamma(\beta) \Gamma(a - \beta) \Gamma(b)}{\Gamma(b-\beta) \Gamma(a)}
\end{equation}
\noindent
is used in the evaluation of $I(a,b,\beta)$. A proof of \eqref{hyper-11} appears in \cite{dixit-2015f}. 

\smallskip

The first evaluation of \eqref{mellin-tricomi} uses the hypergeometric representation \eqref{hyper-tricomi} and 
the formula \eqref{hyper-11}.  This is a traditional computation. Direct substitution gives 
\begin{eqnarray*}
I(a,b,\beta) & = & \frac{\Gamma(b-1)}{\Gamma(a)} 
\int_{0}^{\infty} x^{\beta - b} \pFq11{1+a-b}{2-b}{x} \, dx  +  \\
& & \quad \quad  \frac{\Gamma(1-b)}{\Gamma(1+a-b)} 
\int_{0}^{\infty} x^{\beta-1} \pFq11{a}{b}{x} \, dx \\
& = & -(-1)^{-\beta+b} \frac{\Gamma(b-1)}{\Gamma(a)} 
\frac{\Gamma(\beta-b+1) \Gamma(a- \beta) \Gamma(2-b)}{\Gamma(1+a-b) \Gamma(1- \beta)} \\
& & \quad + (-1)^{-\beta} 
\frac{\Gamma(1-b) }{\Gamma(1+a-b)} \frac{\Gamma(\beta) \Gamma(a - \beta) \Gamma(b)}{\Gamma(b- \beta) 
\Gamma(a)}. 
\end{eqnarray*}
\noindent 
The  result
\begin{equation}
\label{mellin-U}
\int_{0}^{\infty} x^{\beta-1} U(a,b,x) \, dx = 
\frac{\Gamma(a - \beta) \Gamma(\beta - b + 1) \Gamma(\beta)}{\Gamma(a) \Gamma(a-b+1)}
\end{equation}
\noindent
follows from simplification of the previous expression.

\smallskip

The second evaluation of \eqref{mellin-tricomi} uses the method of brackets and 
the divergent series $U_{3}$. It  produces 
the result directly. Start with 
\begin{eqnarray*}
I(a,b,\beta) & = & \int_{0}^{\infty} x^{\beta-1} U(a,b,x) \, dx \\
& = & \int_{0}^{\infty} x^{\beta - a -1} \pFq20{a \,\, \, 1 + a - b }{-}{- \frac{1}{x}} \, dx \\
& = & \sum_{n} \phi_{n} (a)_{n} (1+a-b)_{n} \langle \beta - a - n \rangle.
\end{eqnarray*}
\noindent
A standard evaluation by the method of brackets now reproduces  \eqref{mellin-U}.
\end{example}

\begin{example}
The  evaluation of 
\begin{equation}
J(a,b;\mu) = \int_{0}^{\infty} e^{-  \mu x } U(a,b,x) \, dx
\end{equation}
\noindent
is given next. Start with the expansions 
\begin{equation}
\exp(- \mu x) = \sum_{n_{1}} \phi_{n_{1}} \mu^{n_{1}} x^{n_{1}}
\label{exp-bracket}
\end{equation}
\noindent
and 
\begin{eqnarray*}
U(a,b,x) & = & x^{-a} \,  \pFq20{a \,\,\, 1+a-b}{-}{- \frac{1}{x}} \\
 & = & \frac{x^{-a}}{\Gamma(a) \Gamma(1+a-b)} 
 \sum_{n_{2}} \phi_{n_{2}} \Gamma(a+ n_{2}) \Gamma(1+a-b+n_{2}) x^{-n_{2}}, \nonumber
 \end{eqnarray*}
 \noindent
 to write 
 \begin{equation*}
 J(a,b; \mu) = \frac{1}{\Gamma(a) \Gamma(1+a-b)} 
 \sum_{n_{1},n_{2}} \phi_{n_{1},n_{2}}\mu^{n_{1} }\Gamma(a+ n_{2}) \Gamma(1+a-b+n_{2}) 
 \langle n_{1}-a -n_{2} + 1 \rangle.
 \end{equation*}
 \noindent
 This yields  the two series 
 \begin{eqnarray*}
 J_{1}(a,b;\mu) & = & \frac{1}{\Gamma(a) \Gamma(1+a-b)} 
 \sum_{n} \phi_{n} \Gamma(a-1-n) \Gamma(n+1) \Gamma(2-b+n) \mu^{n} \\
 & = & \frac{\Gamma(2-b)}{(a-1) \Gamma(1+a-b)}
\pFq21{1 \,\,\, 2-b}{2-a}{\mu},  \nonumber
\end{eqnarray*}
\noindent
and 
\begin{eqnarray*}
J_{2}(a,b;\mu) & = & \frac{\mu^{a-1}}{\Gamma(a) \Gamma(1+a-b)} 
\sum_{n} \phi_{n} \Gamma(-a+1-n) \Gamma(a+n) \Gamma(1+a-b+n) \mu^{n} \\
& = & \mu^{a-1} \Gamma(1-a) \pFq10{1+a-b}{-}{\mu} \\
& = & \frac{\mu^{a-1} \, \Gamma(1-a)}{(1-\mu)^{1+a-b}}.
\end{eqnarray*}

\smallskip

In the case $| \mu | < 1$,  both $J_{1}$ and $J_{2}$ are convergent. Therefore 
\begin{equation*}
\int_{0}^{\infty} \exp(- \mu x) U(a,b,x) \, dx = 
 \frac{\Gamma(2-b)}{(a-1) \Gamma(1+a-b)}
\pFq21{1 \,\,\, 2-b}{2-a}{\mu} + 
 \frac{\mu^{a-1} \, \Gamma(1-a)}{(1-\mu)^{1+a-b}}.
 \end{equation*}

 \smallskip 
 
 In the case $\mu=1$, the series $J_{2}$ diverges, so it is discarded. This produces 
 \begin{equation}
 \int_{0}^{\infty} e^{-x} U(a,b,x) \, dx = \frac{\Gamma(2-b)}{(a-1) \Gamma(1+a-b)} 
 \pFq21{1 \,\,\, 2-b}{2-a}{1}.
 \end{equation}
 \noindent
 Gauss'  value \eqref{gauss-value} gives 
 \begin{equation}
 \int_{0}^{\infty} e^{-x} U(a,b,x) \, dx = \frac{\Gamma(2-b)}{\Gamma(2-b+a)}.
 \end{equation}
 \noindent
 In particular, if $a$ is a positive integer, say $a= k$, then 
 \begin{equation}
 \int_{0}^{\infty} e^{-x} U(k,b,x) \, dx = \frac{1}{(b-2)_{k}}.
 \end{equation}
 \noindent
 This result is summarized next.
 
 \begin{proposition}
 Let 
 \begin{equation}
 J(a,b;\mu) = \int_{0}^{\infty} e^{-\mu x} U(a,b,x) \, dx.
 \end{equation}
 \noindent
 Then, for $| \mu |< 1$, 
 \begin{equation}
 J(a,b,\mu) = \frac{\Gamma(2-b)}{(a-1) \Gamma(1+a-b)}
\pFq21{1 \,\,\, 2-b}{2-a}{\mu} + 
 \frac{\mu^{a-1} \, \Gamma(1-a)}{(1-\mu)^{1+a-b}},
 \end{equation}
 \noindent
 and for $\mu=1$,
 \begin{equation}
 J(a,b;1) = \frac{\Gamma(2-b)}{\Gamma(2-b+a)}.
 \end{equation}
 \noindent
 In the special case $a=k \in \mathbb{N}$,
 \begin{equation}
 J(k,b;1) = \frac{1}{(b-2)_{k}}.
 \end{equation}
 \end{proposition}
 \end{example}

 \section{The Airy function}
\label{sec-airy}
\setcounter{equation}{0}

The Airy function, defined by the integral representation 
\begin{equation}
\text{Ai}(x) = \frac{1}{\pi} \int_{0}^{\infty} \cos \left( \frac{t^{3}}{3} + x t \right) \, dt
\end{equation}
\noindent
satisfies the equation 
\begin{equation}
\label{airy-ode}
\frac{d^{2}y}{dx^{2}} - x y = 0,
\end{equation}
\noindent
and the condition $y \to 0$ as $x \to \infty$. A second linearly independent solution of \eqref{airy-ode} is usually 
taken to be 
\begin{equation}
\text{Bi}(x) = \frac{1}{\pi} \int_{0}^{\infty} 
\left[ \exp\left( - \frac{t^{3}}{3} + x t \right) + \sin \left(  \frac{t^{3}}{3} + x t \right)  \right]\, dt.
\end{equation}

Using  \eqref{cosine-hyper} produces 
\begin{eqnarray*}
\text{Ai}(x) & = & \frac{1}{\pi} \sum_{n_{1}} \phi_{n} \frac{1}{\left( \tfrac{1}{2} \right)_{n_{1}} \, 2^{2n_{1}}} 
\int_{0}^{\infty} \left( \frac{t^{3}}{3} + x t \right)^{2n_{1}} \, dt \\
& = & \frac{1}{\pi} \sum_{n_{1},n_{2},n_{3}} \phi_{n_{1},n_{2},n_{3}} \frac{x^{n_{2}} \langle -2n_{1}+n_{2} + n_{3} \rangle }
{ \left( \tfrac{1}{2} \right)_{n_{1}} \, 2^{2n_{1}} \, \Gamma(-2n_{1}) 3^{n_{3}}} \int_{0}^{\infty} t^{3 n_{3} + n_{2}} \, dt \nonumber \\
& = & \sum_{n_{1},n_{2},n_{3}} \phi_{n_{1},n_{2},n_{3}} 
\frac{x^{n_{2}}}{\sqrt{\pi} \Gamma(-2n_{1}) \Gamma \left( \tfrac{1}{2} + n_{1} \right)  2^{2n_{1}} 3^{n_{3}} }
\langle -2n_{1} + n_{2} + n_{3} \rangle \, \langle 3 n_{3}  + n_{2} + 1 \rangle. \nonumber 
\end{eqnarray*}

The usual resolution of this bracket series gives three cases: 

\noindent
\begin{equation}
T_{1} = \frac{1}{2} \sqrt{ \frac{3}{\pi}} \sum_{n=0}^{\infty} \frac{(-1)^{n}}{n!} \frac{\Gamma(- \tfrac{1}{2} - 3n) }
{\Gamma(-2n)} \left( \frac{3}{4} \right)^{n} x^{3n+ 1/2}
\end{equation}
\noindent
a totally null series,
\begin{equation}
T_{2} = \frac{1}{6^{2/3} \, \sqrt{\pi}} \sum_{n=0}^{\infty} \frac{(-1)^{n}}{n!} 
\frac{\Gamma( \tfrac{1}{6} - \tfrac{n}{3} ) }{\Gamma( \tfrac{1}{3} - \tfrac{2n}{3} )} 
\left( \frac{3}{4} \right)^{n/3} x^{n}
\end{equation}
\noindent
a partially divergent series  (at the index $n = 18$), and 

\smallskip

\begin{equation}
T_{3}  = \frac{1}{\sqrt{\pi}} \sum_{n=0}^{\infty} \frac{(-1)^{n}}{n!} 
\frac{\Gamma(3n+1) \Gamma(n+ \tfrac{1}{2})}{\Gamma(-n) \Gamma(2n+1)} 
\left( \frac{4}{3} \right)^{n} x^{-3n-1}
\end{equation}
\noindent
a totally null series, as $T_{1}$ was. 

\begin{example}
The series for $\text{Ai}(x)$ are now used to evaluate the Mellin transform 
\begin{equation}
I(s) = \int_{0}^{\infty} x^{s -1 } \text{Ai}(x) \, dx.
\end{equation}
\noindent
This integral is now computed using the three series $T_{j}$ given above. Using first the value of $T_{1}$
 and the formulas
\begin{equation}
\Gamma(2u) = \frac{2^{2u-1}}{\sqrt{\pi}} \Gamma(u) \Gamma(u + \tfrac{1}{2} ) \text{ and }
\Gamma(3u) = \frac{3^{3u-\tfrac{1}{2}}}{2 \pi} 
\Gamma(u) \Gamma(u + \tfrac{1}{3}) \Gamma( u + \tfrac{2}{3})
\end{equation}
\noindent
(these appear as $8.335.1$ and $8.335.2$ in \cite{gradshteyn-2015a}, respectively), give
\begin{eqnarray}
I(s) & = & \frac{1}{2} \sqrt{ \frac{3}{\pi}} \sum_{n} \phi_{n} 
\left( \frac{3}{4} \right)^{n} \frac{\Gamma( - \tfrac{1}{2} - 3n)}{\Gamma(-2n)} \langle s+ 3 n + \tfrac{1}{2} \rangle \\
& = & \frac{1}{6} \sqrt{\frac{3}{\pi}} \left( \frac{3}{4} \right)^{-s/3 - 1/6} 
\frac{ \Gamma \left( \tfrac{2s  + 1}{6}  \right) \Gamma(s) }{\Gamma \left( \frac{2 s+1}{3}
 \right) }. \nonumber  \\
& = &  3^{-(s+2)/3} \frac{\Gamma(s)}{\Gamma( \frac{s+2}{3})} \nonumber \\
& = &  \frac{3^{(4 s - 7)/6} }{2 \pi} \Gamma \left( \frac{s +1}{3} \right) 
\Gamma\left( \frac{s}{3} \right). \nonumber
\end{eqnarray}
\noindent
Similar calculations, using $T_{2}$ or $T_{3}$, give the same result.  This result is stated next.

\begin{lemma}
The Mellin transform of the Airy function is given by 
\begin{equation}
 \int_{0}^{\infty} x^{s -1 } {\rm{Ai}}(x) \, dx =
  \frac{1}{2 \pi} 3^{(4 s - 7)/6}\Gamma \left( \frac{s +1}{3} \right) 
\Gamma\left( \frac{s}{3} \right). 
\end{equation}
\end{lemma}

\end{example}

\section{The Bessel function $K_{\nu}$}
\label{sec-bessel-nu}
\setcounter{equation}{0}

This section presents series representations for the Bessel function $K_{\nu}(x)$ defined by the integral 
representation
\begin{equation}
K_{\nu}(x) = \frac{2^{\nu} \Gamma(\nu+ \tfrac{1}{2})}{\Gamma(\tfrac{1}{2})} x^{\nu} \int_{0}^{\infty} \frac{\cos t \, dt }
{(x^{2}+t^{2})^{\nu+ \tfrac{1}{2}}},
\end{equation}
\noindent
given as entry $8.432.5$ in \cite{gradshteyn-2015a}. Using the representation \eqref{cosine-hyper} of $\cos t$ as 
$\begin{displaystyle} \pFq01{-}{\tfrac{1}{2}}{- \frac{t^{2}}{4}} \end{displaystyle}$ and using Rule $P_{2}$ in Section 
\ref{sec-method} to expand the 
binomial in the integrand as a bracket series gives 
\begin{equation}
K_{\nu}(x) = 2^{\nu} \sum_{n_{1},n_{2},n_{3}} \phi_{n_{1},n_{2},n_{3}} \frac{x^{2n_{3}+\nu}}{2^{2n_{1}}
\Gamma(n_{1} + \tfrac{1}{2}) }
\langle \nu+ \tfrac{1}{2} + n_{2} + n_{3} \rangle \langle 2 n_{1} + 2n_{2} + 1 \rangle.
\end{equation}
\noindent
The usual procedure to evaluate this bracket series gives three expressions:

\begin{eqnarray}
T_{1} & = & 2^{\nu-1} x^{-\nu} \sum_{n} \phi_{n} \Gamma(\nu - n) \left( \frac{x^{2}}{4} \right)^{n}, \\
T_{2} & = & 2^{-1-\nu} x^{\nu} \sum_{n} \phi_{n} \Gamma(-\nu -n) \left( \frac{x^{2}}{4} \right)^{n}, \nonumber \\
T_{3} & = & 2^{\nu} \sum_{n} \phi_{n} \frac{2^{2n}}{\Gamma(-n)} \Gamma(n+ \nu + \tfrac{1}{2}) 
\Gamma(n + \tfrac{1}{2}) x^{-2n-\nu-1}. \nonumber 
\end{eqnarray}

The series $T_{3}$ is a totally null series for $K_{\nu}$.  In the case $\nu \not \in \mathbb{N}$, the series 
$T_{1}$ and $T_{2}$ are finite  and $K_{\nu}(x) = T_{1}+T_{2}$ gives the usual expression in terms of the 
Bessel $I_{\nu}$ function 
\begin{equation}
K_{\nu}(x) = \frac{\pi}{2} \frac{I_{- \nu}(x) - I_{\nu}(x)}{\sin \pi \nu},
\end{equation}
\noindent
as given in entry $8.485$ in \cite{gradshteyn-2015a}. 

In the case $\nu = k \in \mathbb{N}$, the series $T_{1}$ is partially divergent (the terms $n=0, \, 1, \ldots, k$ 
have divergent coefficients) and the series $T_{2}$ is totally divergent (every coefficient is divergent). In the case 
$\nu=0$, both the series $T_{1}$ and $T_{2}$ become
\begin{equation}
\label{divergent-k0}
\text{Totally divergent series for } K_{0}(x) = \frac{1}{2} \sum_{n} \phi_{n} \Gamma(- n) \left( \frac{x^{2}}{4} \right)^{n},
\end{equation}
\noindent 
using Rule $E_{4}$ to keep a  single copy of the divergent series. This complements the
\begin{equation}
\label{null-k0} 
\text{Totally null series for } K_{0}(x) 
=  \sum_{n} \phi_{n} \frac{2^{2n}}{\Gamma(-n)} \Gamma^{2}(n+ \tfrac{1}{2})  x^{-2n-1}. 
\end{equation}

\medskip

The examples presented below illustrate the use of these divergent series in the computation of definite 
integrals with the Bessel function $K_{0}$ in the  integrand. Entries in \cite{gradshteyn-2015a} 
 with $K_{0}$ as the result of an integral have been discussed in \cite{glasser-2012a}. 

\begin{example}
\label{ex-k0-1}
Entry $6.511.12$ of \cite{gradshteyn-2015a} states that 
\begin{equation}
\int_{0}^{\infty} K_{0}(x) \, dx = \frac{\pi}{2}.
\label{value-k0-1}
\end{equation}
\noindent
To verify this result, use the totally null representation \eqref{null-k0} to obtain 
\begin{eqnarray}
\int_{0}^{\infty} K_{0}(x) \, dx & = & \sum_{n} \phi_{n} \frac{\Gamma \left( n + \tfrac{1}{2} \right)^{2}}{\Gamma(-n)} 
4^{n} \int_{0}^{\infty} x^{-2n-1} \, dx \\
& = & \sum_{n} \phi_{n} \frac{\Gamma \left( n + \tfrac{1}{2} \right)^{2}}{\Gamma(-n)} 
4^{n} \langle -2n \rangle. \nonumber
\end{eqnarray}
\noindent
The value of the bracket series is 
\begin{eqnarray}
\int_{0}^{\infty} K_{0}(x) \, dx & = &  \frac{1}{2} \Gamma \left(n + \tfrac{1}{2} \right)^{2} 4^{n} \Big{|}_{n=0}  \\ 
 & = & \frac{\pi}{2}.  \nonumber 
 \end{eqnarray}
\end{example}

\begin{example}
The Mellin transform 
\begin{equation}
G(\beta,s) = \int_{0}^{\infty} x^{s-1} K_{0}(\beta x) \, dx 
\end{equation}
\noindent
is evaluated next.  Example \ref{ex-k0-1} corresponds to the special case $s=\beta = 1$. The totally divergent 
series \eqref{divergent-k0}  yields
\begin{equation}
G(\beta,s) = \frac{1}{2} \sum_{n} \phi_{n} \Gamma(-n) \frac{\beta^{2n}}{2^{2n}} \langle 2n + s \rangle
\end{equation}
\noindent
and a direct evaluation of the brackets series using Rule $E_{1}$ gives 
\begin{equation}
G(\beta, s) = \frac{2^{s-2}}{\beta^{s}} \Gamma^{2} \left( \frac{s}{2} \right).
\label{mellin-k0}
\end{equation}

Now using the totally null representation \eqref{null-k0} gives the bracket series 
\begin{equation}
G(\beta, s) = \sum_{n} \phi_{n} \frac{2^{2n} \Gamma^{2}(n + \tfrac{1}{2})}{\beta^{2n+1} \, \Gamma(-n)} 
\langle s -1-2n \rangle.
\end{equation}
\noindent
One more application of Rule $E_{1}$ gives \eqref{mellin-k0} again.
\end{example}

\begin{example}
Entry $6.611.9$  of \cite{gradshteyn-2015a} is 
\begin{equation}
\label{formula-6-611-9}
\int_{0}^{\infty} e^{-ax} K_{0}(bx) \, dx  = \frac{1}{\sqrt{b^{2}-a^{2}}} \cos^{-1}\left(  \frac{a}{b} \right),
\end{equation}
\noindent
for $\realpart{(a+b)} > 0$.  This is a generalization of Example \ref{ex-k0-1}. The  totally divergent representation \eqref{divergent-k0} 
and the series for the  exponential function \eqref{exp-bracket} give the bracket series 
\begin{equation}
\int_{0}^{\infty} e^{-ax} K_{0}(bx) \, dx  =  
\frac{1}{2} \sum_{n_{1},n_{2}}  \phi_{n_{1}n_{2}}  \Gamma(-n_{2}) \frac{a^{n_{1}}b^{2n_{2}}}{2^{2n_{2}}}
\langle n_{1} + 2n_{2} + 1 \rangle.
\end{equation}
The usual procedure gives two expressions: 

\noindent
\begin{equation}
T_{1} = \frac{1}{2a} \sum_{n} \phi_{n} \Gamma(2n+1) \Gamma(-n) \left( \frac{b^{2}}{4a^{2}} \right)^{n},
\end{equation}
\noindent
which  is discarded since it is divergent and 
\begin{equation}
T_{2} = \frac{1}{2b} \sum_{n=0}^{\infty} \frac{\Gamma \left( \frac{n+1}{2} \right)^{2}}{n!} 
\left( - \frac{2a}{b} \right)^{n}.
\end{equation}
\noindent
Separating the series according to the parity of the index $n$ yields 
\begin{equation}
T_{2} = \frac{1}{2b} \left[ \pi \sum_{n=0}^{\infty} 
 \frac{ \left( \tfrac{1}{2} \right)_{n}}{n! }\left(  \frac{a^{2}}{b^{2}} \right)^{n} - 
\frac{2a}{b} \sum_{n=0}^{\infty} \frac{(1)_{n}^{2}}{n! \, \left( \frac{3}{2} \right)_{n}} \left( \frac{a^{2}}{b^{2}} \right)^{n} 
\right].
\end{equation}
\noindent
The identity  \cite[$9.121.1$]{gradshteyn-2015a} 
\begin{equation}
\pFq21{-n,b}{b}{-z} = (1+z)^{n},
\end{equation}
\noindent
with $n = -\tfrac{1}{2}$ gives 
\begin{equation}
\frac{\pi}{2b} \sum_{n=0}^{\infty} \frac{ \left( \tfrac{1}{2} \right)_{n}}{n!} \left(  \frac{a^{2}}{b^{2}} \right)^{n} = 
\frac{\pi}{2} \frac{1}{\sqrt{b^{2}-a^{2}}}.
\end{equation}
\noindent
The identity 
\begin{equation}
- \frac{a}{b^{2}} \sum_{n=0}^{\infty} \frac{(1)_{n}^{2}}{n! \, \left( \tfrac{3}{2} \right)_{n}} \left( \frac{a}{b} \right)^{2n} 
= - \frac{1}{\sqrt{b^{2}-a^{2}}} \sin^{-1} \left( \frac{a}{b} \right)
\end{equation}
\noindent
comes from the  Taylor series 
\begin{equation}
\frac{2x \sin^{-1}x}{\sqrt{1-x^{2}}} = \sum_{n=1}^{\infty} \frac{2^{2n}x^{2n}}{n \, \binom{2n}{n}}.
\end{equation}
(See Theorem $7.6.2$ in \cite{moll-2012a} for a proof). The usual argument now gives 
\begin{equation}
T_{2} = \int_{0}^{\infty} e^{-ax} K_{0}(bx) \, dx = \frac{1}{\sqrt{b^{2}-a^{2}}} \left[ \frac{\pi}{2} - 
\sin^{-1}\left( \frac{a}{b} \right) \right],
\end{equation}
\noindent
an equivalent form of \eqref{formula-6-611-9}.
\end{example}

\begin{example}
The next  example,
\begin{equation}
\int_{0}^{\infty} x \sin(bx) K_{0}(ax) \, dx = \frac{\pi b}{2} (a^{2}+b^{2})^{-3/2},
\end{equation}
\noindent
appears as  entry $6.691$ in \cite{gradshteyn-2015a}.  The factor $\sin bx$ in integrand is expressed as a series:
\begin{eqnarray}
\sin(bx) & =& b \, x \, 
 \pFq01{-}{\tfrac{3}{2}}{- \frac{b^{2}x^{2}}{4}} \label{sin-1} \\
 & = & b \Gamma \left( \tfrac{3}{2} \right) \sum_{n_{2}} \phi_{n_{2}} 
 \frac{ \left( \frac{b^{2}}{4} \right)^{n_{2}}}{\Gamma \left( n_{2}  + \tfrac{1}{2} \right)} x^{2n_{2}+1} \nonumber 
\end{eqnarray}
\noindent
and  the Bessel factor is replaced by its totally-null representation \eqref{null-k0}
\begin{equation}
K_{0}(ax) = \frac{1}{a} \sum_{n_{1}} \phi_{n_{1}} \frac{\Gamma \left( n_{1} + \tfrac{1}{2} \right)^{2}}{\Gamma(-n_{1})}
\left( \frac{4}{a^{2}} \right)^{n_{1}} x^{-2n_{1}-1}.
\end{equation}
\noindent
This yields 
\begin{multline}
\int_{0}^{\infty} x \sin(bx) K_{0}(ax) \, dx = \\
\Gamma \left( \frac{3}{2} \right) \sum_{n_{1},n_{2}} \phi_{n_{1},n_{2}} 
\frac{\Gamma \left( n_{1} + \tfrac{1}{2} \right)^{2}}{\Gamma \left( n_{2} + \tfrac{3}{2} \right) \Gamma(-n_{1})} 
\frac{4^{n_{1}-n_{2}} b^{2n_{2}+1}}{a^{2n_{1}+1}} \langle 2+ 2n_{2}-2n_{1} \rangle.
\end{multline}
\noindent
These representation produces two solutions $S_{1}$ and $S_{2}$, one per free index, that \textit{are identical}.  The 
method of brackets rules state that one only should be taken. This is:
\begin{equation}
S_{1}  = \frac{\sqrt{\pi} \, b}{a^{3}} 
\sum_{k=0}^{\infty} \frac{ \Gamma \left( k+ \tfrac{3}{2} \right) (-1)^k b^{2k}  }{k! \, a^{2k}}.
\end{equation}
\noindent
The result now follows from the identity 
\begin{equation}
\sum_{k=0}^{\infty} \frac{ \left( \tfrac{3}{2} \right)_{k}}{k!} \left( - \frac{b}{a} \right)^{k} = 
\pFq10{\tfrac{3}{2}}{-}{- \frac{b}{a}}
\end{equation}
\noindent
and the binomial theorem obtaining
\begin{equation}
\pFq10{\tfrac{3}{2}}{-}{x} = \frac{1}{(1-x)^{3/2}}.
\end{equation}
\end{example}

\begin{example}
The next example in this section evaluates 
\begin{equation}
G(a,b) = \int_{0}^{\infty} J_{0}(ax) K_{0}(bx) \, dx.
\end{equation}
\noindent
From the representation 
\begin{equation}
J_{0}(ax) = \sum_{n_{1}} \phi_{n_{1}} \frac{a^{2n_{1}} x^{2n_{1}}}{2^{2n_{1}} \Gamma(n_{1}+1)}
\end{equation}
\noindent
and the null-series \eqref{k0-null} it follows that 
\begin{equation}
G(a,b) = \sum_{n_{1},n_{2}} \phi_{n_{1},n_{2}} 
\frac{a^{2n_{1}} 2^{2(n_{2}-n_{1})} \Gamma^{2}(n_{2}+ \tfrac{1}{2})}{\Gamma(n_{1}+1) \Gamma(-n_{2}) b^{2n_{2}+1}}
\langle 2n_{1} - 2n_{2} \rangle.
\end{equation}
\noindent
This bracket series generates two identical series, so only one  is kept  to produce 
\begin{eqnarray}
G(a,b) & = & \frac{1}{2b} \sum_{n} \phi_{n} \frac{\Gamma^{2}(n+ \tfrac{1}{2})}{\Gamma(n+1)} 
\left( \frac{a^{2}}{b^{2}} \right)^{n} \\
& = & \frac{\pi}{2b} \pFq21{\frac{1}{2} \,\,\, \frac{1}{2}}{1}{ - \frac{a^{2}}{b^{2}} } \nonumber \\
& = & \frac{1}{b} \mathbf{K} \left( \frac{i a}{b} \right). \nonumber
\end{eqnarray}
Here $\mathbf{K}(z)$ is the elliptic integral of the first kind. Using the identity 
\begin{equation}
\mathbf{K}(i z) = \frac{1}{\sqrt{z^{2}+1}} \mathbf{K} \left( \frac{z}{\sqrt{z^{2}+1}} \right)
\end{equation}
\noindent
yields 
\begin{equation}
G(a,b) = \frac{1}{\sqrt{a^{2}+b^{2}}} \mathbf{K} \left( \frac{a}{\sqrt{a^{2}+b^{2}}} \right).
\end{equation}
\end{example}

\begin{example}
The next example evaluates 
\begin{equation}
H(a) = \int_{0}^{\infty} K_{0}^{2}(ax) \, dx.
\end{equation}
\noindent
Naturally $H(a) = H(1)/a$, but it is convenient to keep $a$ as a parameter.  The problem is generalized to 
\begin{equation}
H_{1}(a,b) = \int_{0}^{\infty} K_{0}(ax)K_{0}(bx) \, dx,
\end{equation}
\noindent
and $H(a) = H_{1}(a,a)$. The evaluation uses the totally divergent  series \eqref{divergent-k0}
\begin{equation}
K_{0}(ax) =  \sum_{n_{1}} \phi_{n_{1}} \frac{a^{2n_{1}} \Gamma(-n_{1})}{2^{2n_{1}+1}} x^{2n_{1}}
\end{equation}
\noindent 
as well as the integral representation (see $8.432.6$ \cite{gradshteyn-2015a}) and the corresponding bracket series
\begin{eqnarray}
K_{0}(bx) & = &  \frac{1}{2} \int_{0}^{\infty} \exp \left( -t - \frac{b^{2}x^{2}}{4t} \right) \, \frac{dt}{t} \\
 & = &  \sum_{n_{2},n_{3}} \phi_{n_{2},n_{3}} 
 \frac{b^{2n_{3}} x^{2n_{3}}}{2^{2n_{3}+1}} \langle n_{2} - n_{3} \rangle.  \nonumber 
 \end{eqnarray}
 \noindent
 Then
 \begin{equation}
 H_{1}(a,b) = \sum_{n_{1},n_{2},n_{3}} \phi_{n_{1},n_{2},n_{3}} 
 \frac{a^{2n_{1}} b^{2n_{3}} \Gamma(-n_{1})}{2^{2n_{1}+2n_{3}+2}} 
 \langle n_{2} - n_{3} \rangle \, \langle 2n_{1} + 2n_{3} + 1 \rangle.
 \end{equation}
 \noindent
 The evaluation of this bracket series requires an extra parameter $\varepsilon$ and to consider 
 \begin{equation}
 H_{2}(a,b,\varepsilon) = \sum_{n_{1},n_{2},n_{3}} \phi_{n_{1},n_{2},n_{3}} 
 \frac{a^{2n_{1}} b^{2n_{3}} \Gamma(-n_{1})}{2^{2n_{1}+2n_{3}+2}} 
 \langle n_{2} - n_{3} + \varepsilon \rangle \, \langle 2n_{1} + 2n_{3} + 1 \rangle.
 \end{equation}
 \noindent
 Evaluating this brackets series produces three values, one divergent, which is discarded, and two others:
 \begin{eqnarray}
 T_{2} & = &  \frac{1}{4a} c^{\varepsilon} \sum_{n} \phi_{n} \Gamma(-n- \varepsilon) 
 \Gamma^{2}(\varepsilon + n + \tfrac{1}{2})c^{n} \\
 T_{3} & = & \frac{1}{4a} \sum_{n} \phi_{n} \Gamma(-n + \varepsilon) \Gamma^{2}(n + \tfrac{1}{2}) c^{n},
 \nonumber 
 \end{eqnarray}
 \noindent
 with $c = b^{2}/a^{2}$. Converting the $\Gamma$-factors into Pochhammer symbols  produces 
 \begin{eqnarray}
 T_{2} & = &  \frac{1}{4a} c^{\varepsilon} \Gamma(- \varepsilon) \Gamma^{2} \left( \tfrac{1}{2} + \varepsilon \right) 
 \pFq21{\tfrac{1}{2} + \varepsilon \,\,\, \tfrac{1}{2} + \varepsilon }{1+ \varepsilon}{\,\,\, c} \\
 T_{3} & = & \frac{\pi}{4a} \Gamma(\varepsilon)  \,\, 
  \pFq21{\tfrac{1}{2}   \,\,\, \tfrac{1}{2}   }{1 -\varepsilon}{\,\,\, c}.
 \nonumber 
 \end{eqnarray}
 \noindent
 This yields 
 \begin{equation*}
 H_{2}(a,b, \varepsilon) = \frac{\pi}{4a} 
 \left[ \Gamma(\varepsilon)  \pFq21{\tfrac{1}{2}  \,\,\, \tfrac{1}{2}  }{1 - \varepsilon}{\,\,\, c}  - 
 c^{\varepsilon} 
 \frac{\Gamma^{2}( \tfrac{1}{2} + \varepsilon)}{\varepsilon \,\Gamma(\varepsilon) \sin \pi \varepsilon}
  \pFq21{\tfrac{1}{2} + \varepsilon \,\,\, \tfrac{1}{2} + \varepsilon }{1+ \varepsilon}{\,\,\, c} \right].
  \end{equation*}
  Let $c \to 1 \,\,  ( b \to a)$ and use Gauss'  formula \eqref{gauss-value} to obtain 
  \begin{equation*}
  \pFq21{\frac{1}{2} \,\,\, \frac{1}{2} }{1- \varepsilon}{1} = \frac{\Gamma(1 - \varepsilon) \Gamma(- \varepsilon)}
  {\Gamma^{2} \left( \frac{1}{2} - \varepsilon \right)}
 \,\,\, {\rm  and } \,\,\,
  \pFq21{\frac{1}{2} + \varepsilon  \,\,\, \frac{1}{2}  + \varepsilon }{1+ \varepsilon}{1} =
   \frac{\Gamma(1 + \varepsilon) \Gamma(- \varepsilon)} {\Gamma^{2} \left( \frac{1}{2}  \right)},
  \end{equation*}
  \noindent
  and this produces 
  \begin{eqnarray*}
  H_{2}(a,a,\varepsilon) & = & \frac{\Gamma(- \varepsilon)^{2} \Gamma^{2} \left( \varepsilon + \tfrac{1}{2} \right) 
  \Gamma( \varepsilon + 1)}{4 \pi a} +  \frac{\pi \Gamma(1 - \varepsilon) \Gamma(- \varepsilon) 
  \Gamma(\varepsilon)}{4 a \, \Gamma^{2} \left( \tfrac{1}{2} - \varepsilon \right)}  \\
  & = & \frac{\pi}{4a} 
  \left[ \frac{\Gamma^{2}(- \varepsilon) \Gamma(\varepsilon+1) \Gamma^{2}(\varepsilon + \tfrac{1}{2})}{\pi^{2}} 
  + \frac{\Gamma(1 - \varepsilon) \Gamma(- \varepsilon) \Gamma(\varepsilon)}{\Gamma^{2}( \tfrac{1}{2} 
  - \varepsilon)} \right]. \nonumber
  \end{eqnarray*}
  \noindent
 Expanding $H_{2}(a,a,\varepsilon)$ in powers of $\varepsilon$ gives 
 \begin{equation}
 H(a,a,\varepsilon) = \frac{\pi^{2}}{4a} - \frac{\pi^{2}}{4a} ( \gamma + 4 \ln 2 ) \varepsilon + 
 o(\varepsilon).
 \end{equation}
 \noindent
 Letting $\varepsilon \to 0$ gives
 \begin{equation}
 \int_{0}^{\infty} K_{0}^{2}(ax) \, dx = \frac{\pi^{2}}{4a}.
 \label{k0-squared}
 \end{equation}
\end{example}

\begin{example}
The final example in this section is the general integral 
\begin{equation}
I(a,b;\nu,\lambda;\rho)  = \int_{0}^{\infty} x^{\rho-1} K_{\nu}(ax) K_{\lambda}(bx) \, dx.
\label{int-kgen}
\end{equation}
\noindent
The case $a=b$ appears in \cite{kolbig-1995a}. 

The evaluation uses the integral representation 
\begin{equation}
K_{\nu}(ax) = \frac{(ax)^{\nu}}{2^{\nu+1}} \int_{0}^{\infty} \text{exp}\left( - t - \frac{a^{2}x^{2}}{4t} \right) \frac{dt}{t^{\nu+1}}
\end{equation}
\noindent
appearing in \cite[$8.432.6$]{gradshteyn-2015a}. This produces the bracket series representation 
\begin{equation}
K_{\nu}(ax) = \frac{1}{2^{\nu+1}} \sum_{n_{1},n_{2}} \phi_{n_{1},n_{2}} 
\frac{a^{2n_{2}+\nu}}{2^{2n_{2}}} x^{2n_{2}+\nu} \langle n_{1}-n_{2} - \nu \rangle.
\end{equation}
\noindent
The second factor uses the totally null representation \eqref{k0-null}
\begin{equation}
K_{\lambda}(bx) = 2^{\lambda} \sum_{n_{3}} \phi_{n_{3}} 
\frac{2^{2n_{3}} \Gamma(n_{3} + \lambda + \tfrac{1}{2} ) \Gamma(n_{3} + \tfrac{1}{2}) }{\Gamma(-n_{3}) 
b^{2n_{3} + \lambda + 1} } \frac{1}{x^{2n_{3}+\lambda + 1}}.
\end{equation}
\noindent
Replacing in \eqref{int-kgen} produces the bracket series 
\begin{multline}
I(a,b;\nu,\lambda;\rho) = \sum_{n_{1},n_{2},n_{3}} \phi_{n_{1},n_{2},n_{3}}
\frac{a^{2n_{2}+\nu} 2^{\lambda - \nu -1 + 2n_{3} - 2n_{2}} \Gamma(n_{3}+ \lambda + \tfrac{1}{2}) 
\Gamma(n_{3} + \tfrac{1}{2} )}{b^{2n_{3} + \lambda +1} \, \Gamma(-n_{3})} \label{nice-brackets} \\  
\times \langle n_{1}-n_{2} - \nu \rangle \langle \rho + \nu - \lambda + 2n_{2} - 2n_{3} - 1 \rangle.
\end{multline}
\noindent 
The vanishing of the brackets gives the system of equations
\begin{eqnarray}
n_{1} - n_{2} & = & \nu \\
2n_{2} - 2n_{3} & = & -\rho - \nu + \lambda +1. \nonumber 
\end{eqnarray}
\noindent
The matrix of coefficients is of rank $2$, so it produces  three series as candidates for values of the integral, one 
per free index. 

\smallskip

\noindent
\textit{Case 1}: $n_{1}$ free. Then $n_{2} = n_{1} - \nu$ and $n_{3} = \frac{\rho - \nu - \lambda-1}{2} + n_{1}$. This 
gives 
\begin{equation*}
T_{1} = 2^{\rho-3} \frac{b^{\nu- \rho} }{a^{\nu}}
\Gamma \left( \frac{\rho - \nu + \lambda}{2} \right) 
\Gamma \left( \frac{\rho - \nu - \lambda}{2} \right)
\Gamma(\nu) 
\, \pFq21{ \frac{\rho-\nu + \lambda}{2} \,\, \frac{\rho - \nu - \lambda}{2}}{1 - \nu}{ \frac{a^{2}}{b^{2}}}.
\end{equation*}

\smallskip 

\noindent
\textit{Case 2}: $n_{2}$ free. Then $n_{1} = n_{2} + \nu$ and $n_{3} = \frac{\rho+\nu - \lambda - 1}{2} 
+ n_{2}$. This gives 
\begin{equation*}
T_{2} = 2^{\rho-3} \frac{a^{\nu}}{b^{\nu+\rho}} \Gamma(- \nu) 
\Gamma \left( \frac{\rho +\nu + \lambda}{2} \right) 
\Gamma \left( \frac{\rho + \nu - \lambda}{2} \right)
\Gamma(\nu) 
\, \pFq21{ \frac{\rho + \nu + \lambda}{2} \,\, \frac{\rho + \nu - \lambda}{2}}{1 + \nu}{ \frac{a^{2}}{b^{2}}}.
\end{equation*}

\smallskip 

\noindent
\textit{Case 3}: $n_{3}$ free. Then $n_{2} = n_{3} + \frac{\lambda - \rho - \nu +1}{2}$ and 
$n_{1} = n_{3} + \frac{\lambda - \rho + \nu +1}{2}$. This produces 
\begin{multline*}
T_{3} = 2^{\rho - 3} \frac{a^{-\rho + \lambda +1}}{b^{\lambda + 1}} 
\sum_{n} \frac{\phi_{n} }{\Gamma(-n)}
\Gamma \left( \frac{\rho + \nu - \lambda -1}{2} - n \right) 
\Gamma \left( \frac{\rho - \nu - \lambda -1}{2} - n \right)  \\
\Gamma \left( n + \lambda + \tfrac{1}{2} \right) 
\Gamma(n+ \tfrac{1}{2} ) 
\left( \frac{a^{2}}{b^{2}} \right)^{n}.
\end{multline*}
\noindent
This series has the value zero. This proves the next statement:

\begin{proposition}
The integral 
\begin{equation}
I(a,b;\nu,\lambda; \rho) = \int_{0}^{\infty} x^{\rho-1} K_{\nu}(ax) K_{\lambda}(bx) \, dx
\end{equation}
\noindent
is given by 
\begin{multline*}
I(a,b;\nu,\lambda; \rho)  = \\  
2^{\rho-3} \frac{b^{\nu- \rho} }{a^{\nu}}  \Gamma(\nu)
\Gamma \left( \frac{\rho - \nu + \lambda}{2} \right) 
\Gamma \left( \frac{\rho - \nu - \lambda}{2} \right)
\, \pFq21{ \frac{\rho-\nu + \lambda}{2} \,\, \frac{\rho - \nu - \lambda}{2}}{1 - \nu}{ \frac{a^{2}}{b^{2}}} \\
\hspace{0.25in} +  2^{\rho-3} \frac{a^{\nu}}{b^{\nu+\rho}} \Gamma(- \nu) 
\Gamma \left( \frac{\rho +\nu + \lambda}{2} \right) 
\Gamma \left( \frac{\rho + \nu - \lambda}{2} \right)
\, \pFq21{ \frac{\rho + \nu + \lambda}{2} \,\, \frac{\rho + \nu - \lambda}{2}}{1 + \nu}{ \frac{a^{2}}{b^{2}}}.
\end{multline*}
\end{proposition}

Some special cases of this evaluation are interesting in their own right.  Consider first the case $a=b$. Using 
Gauss' theorem \eqref{gauss-value} it follows that 
\begin{equation}
T_{1} = 
\frac{2^{\rho - 3}\,  \Gamma(\nu) \Gamma \left( \frac{\rho + \lambda - \nu}{2} \right) 
\Gamma \left( \frac{\rho - \lambda - \nu}{2} \right)
\Gamma(1- \nu) \Gamma(1- \rho)}
{ a^{\rho} \,\Gamma \left( 1 - \frac{\rho +\nu + \lambda }{2} \right) 
 \Gamma \left( 1 - \frac{\rho+\nu -  \lambda }{2} \right) }
 \end{equation}
 \noindent
 and 
 \begin{equation}
T_{2} = 
\frac{2^{\rho - 3}\,  \Gamma(-\nu) \Gamma \left( \frac{\rho + \lambda + \nu}{2} \right) 
\Gamma \left( \frac{\nu + \rho - \lambda }{2} \right)
\Gamma(\nu+1) \Gamma(1- \rho)}
{ a^{\rho} \,\Gamma \left( 1 - \frac{\rho -\nu - \lambda }{2} \right) 
 \Gamma \left( 1 - \frac{\rho - \nu +  \lambda }{2} \right) }.
 \end{equation}
 
 \begin{proposition}
 The integral 
\begin{equation}
J(a;\nu,\lambda; \rho) = \int_{0}^{\infty} x^{\rho-1} K_{\nu}(ax) K_{\lambda}(ax) \, dx
\end{equation}
\noindent
is given by 
\begin{multline*}
J(a;\nu,\lambda; \rho)  =  
\frac{2^{\rho - 3}\,  \Gamma(\nu) \Gamma \left( \frac{\rho + \lambda - \nu}{2} \right) 
\Gamma \left( \frac{\rho - \lambda - \nu}{2} \right)
\Gamma(1- \nu) \Gamma(1- \rho)}
{ a^{\rho} \,\Gamma \left( 1 - \frac{\rho +\nu + \lambda }{2} \right) 
 \Gamma \left( 1 - \frac{\rho+\nu -  \lambda }{2} \right) } + \\
\frac{2^{\rho - 3}\,  \Gamma(-\nu) \Gamma \left( \frac{\rho + \lambda + \nu}{2} \right) 
\Gamma \left( \frac{\nu + \rho - \lambda }{2} \right)
\Gamma(\nu+1) \Gamma(1- \rho)}
{ a^{\rho} \,\Gamma \left( 1 - \frac{\rho -\nu - \lambda }{2} \right) 
 \Gamma \left( 1 - \frac{\rho - \nu +  \lambda }{2} \right) }.
 \end{multline*}
 \end{proposition}
 
 The next special case is to take $a=b$ and $\lambda = \nu$. Then 
 \begin{equation}
 T_{1} = \frac{2^{\rho-3}}{a^{\rho}} 
 \frac{\Gamma(\nu) \Gamma \left(\frac{\rho}{2} \right) \Gamma \left( \frac{\rho}{2} - \nu \right) 
 \Gamma(1- \nu) \Gamma(1 - \rho)}
 { \Gamma \left( 1 - \frac{\rho}{2} - \nu \right) \Gamma \left( 1 - \frac{\rho}{2} \right)}
 \end{equation}
 \noindent
 and 
 \begin{equation}
 T_{2} = \frac{2^{\rho-3}}{a^{\rho}} 
 \frac{\Gamma(-\nu) \Gamma \left(\frac{\rho}{2} \right) \Gamma \left( \frac{\rho}{2} + \nu \right) 
 \Gamma(\nu+1) \Gamma(1 - \rho)}
 { \Gamma \left( 1 - \frac{\rho}{2} +  \nu \right) \Gamma \left( 1 - \frac{\rho}{2} \right)}.
 \end{equation}
 \noindent
 This proves the next result:
 
 \begin{proposition}
 The integral 
 \begin{equation}
 L(a; \nu, \rho)  =   \int_{0}^{\infty} x^{\rho-1} K_{\nu}^{2}(ax) \, dx 
 \end{equation}
 \noindent
 is given by 
 \begin{eqnarray*}
 L(a; \nu, \rho) 
& = &  \frac{2^{\rho-3}}{a^{\rho}} 
 \left[ \frac{\Gamma(\nu) \Gamma(1- \nu) \Gamma \left( \frac{\rho}{2} - \nu \right) }
 {\Gamma \left( 1 - \frac{\rho}{2} - \nu \right)} 
 + 
 \frac{\Gamma(-\nu) \Gamma(1+\nu) \Gamma \left( \frac{\rho}{2} + \nu \right) }
 {\Gamma \left( 1 - \frac{\rho}{2} + \nu \right)} 
 \right].
 \end{eqnarray*}
 \end{proposition}

\medskip 

The last special case is $\rho=1$; that is, the integral 
\begin{equation}
M(a,b;\nu,\lambda) = \int_{0}^{\infty} K_{\nu}(ax) K_{\lambda}(bx) \, dx.
\end{equation}
\noindent
It is shown that the usual application of the method of brackets yield only divergent series, so a new approach 
is required. 

The argument begins with converting the brackets series in \eqref{nice-brackets}  to 
\begin{multline}
M(a,b;\nu,\lambda) = \sum_{n_{1},n_{2},n_{3}} \phi_{n_{1},n_{2},n_{3}}
\frac{a^{2n_{2}+\nu} 2^{\lambda - \nu -1 + 2n_{3} - 2n_{2}} \Gamma(n_{3}+ \lambda + \tfrac{1}{2}) 
\Gamma(n_{3} + \tfrac{1}{2} )}{b^{2n_{3} + \lambda +1} \, \Gamma(-n_{3})} \label{nice-brackets1} \\  
\times \langle n_{1}-n_{2} - \nu \rangle \langle  \nu - \lambda + 2n_{2} - 2n_{3}  \rangle.
\end{multline}
\noindent
A routine application of the method of brackets gives three series 
\begin{eqnarray*}
T_{1} & = & \frac{b^{\nu-1}}{4a^{\nu}} 
\Gamma \left( \frac{1-\nu+\lambda}{2} \right) \Gamma \left( \frac{1- \nu - \lambda}{2} \right) \Gamma(\nu) 
\pFq21{ \frac{1 - \nu + \lambda}{2} \,\,\, \frac{1 - \nu - \lambda}{2}}{1-\nu}{ \frac{a^{2}}{b^{2}}} \nonumber  \\
T_{2} & = & \frac{a^{\nu}}{4b^{\nu+1}} 
\Gamma \left( \frac{1+\nu+\lambda}{2} \right) \Gamma \left( \frac{1+\nu - \lambda}{2} \right) \Gamma(-\nu) 
\pFq21{ \frac{1 + \nu + \lambda}{2} \,\,\, \frac{\nu - \lambda+1}{2}}{1+\nu}{ \frac{a^{2}}{b^{2}}} \nonumber
\end{eqnarray*}
\noindent
and a totally null series $T_{3}$. Gauss' value \eqref{gauss-value} shows that $T_{1}$ and $T_{2}$ diverge 
when $a \to b$.  Therefore \eqref{nice-brackets1} is replaced by 
\begin{multline}
M(a,b;\nu,\lambda) = \lim\limits_{\varepsilon \to 0} \sum_{n_{1},n_{2},n_{3}} \phi_{n_{1},n_{2},n_{3}}
\frac{a^{2n_{2}+\nu} 2^{\lambda - \nu -1 + 2n_{3} - 2n_{2}} \Gamma(n_{3}+ \lambda + \tfrac{1}{2}) 
\Gamma(n_{3} + \tfrac{1}{2} )}{b^{2n_{3} + \lambda +1} \, \Gamma(-n_{3})} \label{nice-brackets2} \\  
\times \langle n_{1}-n_{2} - \nu  + \varepsilon \rangle \langle  \nu - \lambda + 2n_{2} - 2n_{3}  \rangle.
\end{multline}
\noindent
Proceeding as before produces a null series that is discarded and also 
\begin{eqnarray*}
T_{1} & = & \frac{a^{-\nu+2 \varepsilon}}{4 b^{1 - \nu + 2 \varepsilon}} 
\Gamma(\nu - \varepsilon) \Gamma \left( \frac{1 + \lambda - \nu}{2} + \varepsilon \right)  
\Gamma \left( \frac{1 - \lambda - \nu}{2} + \varepsilon \right)  \\
& & \quad \times  \pFq21{ \frac{1 - \nu + \lambda}{2} + \varepsilon \,\,\, 
\frac{1 - \nu - \lambda}{2} + \varepsilon}{1 - \nu + \varepsilon}{ \frac{a^{2}}{b^{2}}} \\
T_{2} & = & \frac{a^{\nu}}{4 b^{1 + \nu} }
\Gamma(-\nu + \varepsilon) \Gamma \left( \frac{1 + \lambda + \nu}{2} \right) 
\Gamma \left( \frac{1 - \lambda + \nu}{2}  \right)  \\
& & \quad  \times \pFq21{ \frac{1 + \nu - \lambda}{2}  \,\,\, 
\frac{1 + \nu + \lambda}{2}}{1 + \nu - \varepsilon}{ \frac{a^{2}}{b^{2}}}.
\end{eqnarray*}
\noindent
In the limit as $b \to a$, these become 
\begin{eqnarray*}
T_{1} & = &  \frac{\Gamma(\nu- \varepsilon) \Gamma \left( \frac{1 + \lambda - \nu}{2} + \varepsilon \right) 
\Gamma \left( \frac{1 - \lambda - \nu}{2} + \varepsilon \right) 
\Gamma(1 - \nu + \varepsilon ) \Gamma( - \varepsilon) }
{ 4a \, \Gamma \left( \frac{1 - \nu - \lambda}{2}  \right) 
\Gamma \left( \frac{1 - \nu + \lambda}{2}  \right)} \\
T_{2} & = & 
\frac{\Gamma(-\nu+ \varepsilon) \Gamma \left( \frac{1 + \lambda +\nu}{2} \right) 
\Gamma \left( \frac{1 - \lambda + \nu}{2} \right) 
\Gamma(1 + \nu - \varepsilon ) \Gamma( - \varepsilon) }
{ 4a \, \Gamma \left( \frac{1 + \nu + \lambda}{2} - \varepsilon \right) 
\Gamma \left( \frac{1 + \nu - \lambda}{2} - \varepsilon \right)}.
\end{eqnarray*}
\noindent
Passing to the limit as $\varepsilon \to 0$ gives 
\begin{equation}
  \int_{0}^{\infty} K_{\nu}(ax) K_{\lambda}(ax) \, dx 
  =  \frac{\pi^{2}}{4 a \sin \pi \nu} \left[ \tan \left( \frac{\pi}{2}(\lambda + \nu) \right) - 
 \tan \left( \frac{\pi}{2}( \lambda - \nu) \right) \right]. 
 \end{equation}
 \noindent
 In the special case $\lambda = \nu$, it follows that 
 \begin{equation}
 \int_{0}^{\infty} K_{\nu}^{2}(ax) \, dx = \frac{\pi^{2}}{4a \cos \pi \nu }, \text{ valid for } | \nu | < \tfrac{1}{2}.
\end{equation}
\noindent
This value generalizes \eqref{k0-squared}. It appears in Prudnikov et al. \cite{prudnikov-1986a} as 
entries $.2.16.28.3$ and $2.16.33.2$.
\end{example}

\section{An example with an integral producing the Bessel function}
\label{sec-producing}

The evaluation of integrals in Section \ref{sec-bessel-nu} contain the Bessel function $K_{\nu}$ in the 
integrand. This section uses the method developed in the current work to evaluate some entries  in 
\cite{gradshteyn-2015a} where the answer involves  $K_{0}$. 

\begin{example}
\label{ex-rule-e4}
The first example is entry $6.532.4$ in \cite{gradshteyn-2015a}
\begin{equation}
\label{gr-65324}
\int_{0}^{\infty} \frac{x J_{0}(ax)}{x^{2}+b^{2}} \, dx = K_{0}(ab).
\end{equation}
The analysis begins with the series 
\begin{eqnarray}
J_{0}(ax) & = &  \sum_{n=0}^{\infty} \frac{1}{n!^{2}} \left(  - \frac{a^{2}x^{2}}{4} \right)^{n} \\
 & = & \sum_{n_{1}=0}^{\infty} \phi_{n_{1}} \frac{  a^{2n_{1}}} {2^{2n_{1}}  \, \Gamma(n_{1}+ 1)} x^{2n_{1}}
 \nonumber 
 \end{eqnarray}
 
 Rule $P_{2}$ gives 
 \begin{equation}
 \frac{1}{x^{2}+b^{2}} = \sum_{n_{2},n_{3}} \phi_{n_{2},n_{3}} x^{2n_{2}} b^{2n_{3}} 
 \langle 1 + n_{2} + n_{3} \rangle.
 \end{equation}
 \noindent
 Therefore 
 \begin{equation}
 \int_{0}^{\infty} \frac{x J_{0}(ax) }{x^{2}+b^{2}} \, dx  = 
 \sum_{n_{1},n_{2},n_{3}} \phi_{n_{1},n_{2},n_{3}} 
 \frac{a^{2n_{1}} b^{2n_{3}}}{2^{2n_{1}} \, \Gamma(n_{1}+1)} 
 \langle 1+n_{2}+n_{3} \rangle 
 \langle 2 + 2n_{1} + 2n_{2} \rangle.
 \end{equation}
 \noindent
 The method of brackets produces three series as candidates for solutions, one per free index $n_{1}, \, n_{2}, \, 
 n_{3}$:
 \begin{eqnarray}
 T_{1} & = & \frac{1}{2} \sum_{n=0}^{\infty} \phi_{n} \Gamma(-n) \left( \frac{a^{2}b^{2}}{4} \right)^{n} \\
 T_{2} & = & \frac{2}{a^{2}b^{2}} \sum_{n=0}^{\infty} \phi_{n} 
 \frac{\Gamma^{2}(n) }{\Gamma(-n)} \left( \frac{4}{a^{2}b^{2}} \right)^{n}  \nonumber \\
 T_{3} & = &   \frac{1}{2} \sum_{n=0}^{\infty} \phi_{n} \Gamma(-n) \left( \frac{a^{2}b^{2}}{4} \right)^{n}.  \nonumber 
 \end{eqnarray}
 \noindent
 The fact that $T_{1} = T_{3}$ and using Rule $E_{4}$ shows that only one of these series has to be counted. Since 
 $T_{1}$ and $T_{2}$ are non-classical series of \textit{distinct variables}, both are representations of the 
 value of the integral.  Observe 
 that $T_{2}$ is the totally null representation of $K_{0}(ab)$ given in \eqref{k0-divergent}. This confirms 
 \eqref{gr-65324}.  The fact that $T_{3}$ is also a value for the integral gives another totally divergent 
 representation for $K_{0}$:
 \begin{equation}
 \label{new-k0}
 K_{0}(x) = \frac{2}{x^{2}} \sum_{n=0}^{\infty} \phi_{n} \frac{\Gamma^{2}(n+1)}{\Gamma(-n)} 
 \left( \frac{4}{x^{2}} \right)^{n}.
 \end{equation}
 \noindent
 To test its validity, the integral in Example \ref{ex-k0-1} is evaluated again, this time using \eqref{new-k0}:
 \begin{eqnarray}
 \int_{0}^{\infty} K_{0}(x) dx & = & 
 \int_{0}^{\infty} \frac{2}{x^{2}} \sum_{n} \phi_{n} \frac{\Gamma^{2}(n+1)}{\Gamma(-n)} 
 2^{2n} x^{-2n} \\
 & = & \sum_{n} \phi_{n} 2^{2n+1} \frac{\Gamma^{2}(n+1)}{\Gamma(-n)} \int_{0}^{\infty} x^{-2n-2} \, dx \nonumber \\
 & = & \sum_{n} \phi_{n} 2^{2n+1} \frac{\Gamma^{2}(n+1)}{\Gamma(-n)} \langle -2n-1 \rangle.  \nonumber
 \end{eqnarray}
 \noindent
 The bracket series is evaluated using Rule $E_{1}$ to confirm \eqref{value-k0-1}.
 \end{example}
 
 \begin{example}
 Entry $6.226.2$ in \cite{gradshteyn-2015a} is 
 \begin{equation}
 \label{62262}
 \int_{0}^{\infty} {\rm{Ei}}\left(- \frac{a^{2}}{4x} \right) e^{-\mu x} \, dx = - \frac{2}{\mu} K_{0}(a \sqrt{\mu}).
 \end{equation}
 \noindent
 The evaluation starts with the partially divergent series \eqref{pds-ei1}
 \begin{equation}
 {\rm{Ei}}\left( - \frac{a^{2}}{4x} \right) = \sum_{n_{1}=0}^{\infty} \phi_{n_{1}} \frac{a^{2n_{1}}}{n_{1} 2^{2n_{1}}}
 \frac{1}{x^{n_{1}}}
 \end{equation}
 \noindent
 and this yields 
 \begin{equation}
  \int_{0}^{\infty} {\rm{Ei}}\left(- \frac{a^{2}}{4x} \right) e^{-\mu x} \, dx = 
  \sum_{n_{1},n_{2}} \phi_{n_{1}n_{2}} \frac{a^{2n_{1}} \mu^{n_{2}}}{n_{1} 2^{2n_{1}}} \langle n_{2} - n_{1} + 1 \rangle.
  \end{equation}
  \noindent
  The method of brackets gives two series. The first one 
  \begin{eqnarray} 
  T_{1} & = & \frac{1}{\mu} \sum_{n_{1}} \phi_{n_{1}} \frac{\Gamma(1-n_{1})}{n_{1}2^{2n_{1}}} (a^{2} \mu)^{n_{1}}
  \label{form-t1}  \\
  & = & - \frac{1}{\mu} \sum_{n_{1}} \phi_{n_{1}}  \frac{\Gamma(-n_{1}) }{2^{2n_{1}}} (a^{2} \mu)^{n_{1}} \nonumber \\
  & = & - \frac{2}{\mu} K_{0}( a \sqrt{\mu}), \nonumber 
  \end{eqnarray}
  \noindent
 using \eqref{k0-divergent}.  The second series is 
  \begin{equation}
  T_{2}  =   \sum_{n_{2}}\phi_{n_{2}} \frac{a^{2n_{2}+2} \mu^{n_{2}}}{(n_{2}+1) 2^{2(n_{2}+1)}}
  \Gamma(-n_{2}-1).
  \end{equation}
  \noindent
  Now shift the index by $m = n_{2}+1$ to obtain 
  \begin{eqnarray}
  T_{2}  & = & \sum_{m} \phi_{m-1} \frac{a^{2m} \mu^{m-1}}{m2^{2m}}
  \Gamma(-m). \nonumber  \\
  & = & - \frac{1}{\mu} \sum_{m} \phi_{m}  \Gamma(-m) \frac{a^{2m} \mu^{m}}{2^{2m}}. \nonumber 
  \end{eqnarray}
  \noindent
  This is the same sum as $T_{1}$ in the second line of \eqref{form-t1}. Recall  that the 
  summation indices are placed after the conversion of 
  the indicator $\phi_{n_{2}}$ to its expression in terms of the gamma function. According to Rule $E_{4}$, the 
  sum $T_{2}$  is discarded.  This establishes \eqref{62262}. 
 \end{example}

\section{A new use of the method of brackets}
\label{sec-new-use}
\setcounter{equation}{0}

This section introduces a procedure to evaluate integrals of the form 
\begin{equation}
I(a_{1},a_{2}) = \int_{0}^{\infty} f_{1}(a_{1}x) f_{2}(a_{2}x) \, dx.
\end{equation}

Differentiating with respect to the parameters leads to 
\begin{equation}
a_{1} \frac{\partial I(a_{1},a_{2})}{\partial a_{1}} + 
a_{2} \frac{\partial I(a_{1},a_{2})}{\partial a_{2}}  = \int_{0}^{\infty} 
x \frac{d}{dx} \left[ f_{1}(a_{1}x) f_{2}(a_{2} x) \right] \, dx.
\end{equation}
\noindent
Integration by parts produces 
\begin{equation}
\label{formula-ode1}
I(a_{1},a_{2}) = x f_{1}(a_{1}x)f_{2}(a_{2}x)\Big{|}_{0}^{\infty} - 
\left( a_{1} \frac{\partial I(a_{1},a_{2})}{\partial a_{1}} + 
a_{2} \frac{\partial I(a_{1},a_{2})}{\partial a_{2}} \right).
\end{equation}
\noindent
A direct extension to many parameters leads to the following result.

\begin{theorem}
\label{nice-1}
Let 
\begin{equation}
I(a_{1},\cdots,a_{n}) = \int_{0}^{\infty} \prod_{j=1}^{n} f(a_{j}x) \, dx.
\end{equation}
\noindent
Then 
\begin{equation}
\label{nice-form1}
I(a_{1},\cdots, a_{n}) = x \prod_{j=1}^{n} f_{j}(a_{j}x)\Big{|}_{0}^{\infty} - 
\sum_{j=1}^{n}  a_{j} \frac{\partial I(a_{1},\cdots, a_{n})}{\partial a_{j}}.
\end{equation}
\end{theorem}

\begin{example}
The integral 
\begin{equation}
I(a,b) = \int_{0}^{\infty} e^{-ax} J_{0}(bx) \, dx 
\end{equation}
\noindent
is evaluated first by a direct application of the method of brackets and then using Theorem \ref{nice-1}. 

\smallskip

The  bracket series for $I(a,b)$  
\begin{equation}
I(a,b) = \sum_{n_{1},n_{2}} \phi_{n_{1},n_{2}} \frac{a^{n_{1}} b^{2n_{2}}}{2^{2n_{2}} \Gamma(n_{2}+1)}
\langle n_{1} + 2n_{2}+1 \rangle
\end{equation}
\noindent 
is obtained directly from \eqref{exp-bracket}
\begin{equation}
e^{-ax} = \sum_{n_{1}} \phi_{n_{1}} a^{n_{1}} x^{n_{1}}
\end{equation}
\noindent 
and 
\begin{equation}
J_{0}(bx) = \pFq01{-}{1}{- \frac{(bx)^{2}}{4}} = \sum_{n_{2}} \phi_{n_{2}}
\frac{b^{2n_{2}}}{\Gamma(n_{2}+1) 2^{2n_{2}}}
x^{2n_{2}}.
\end{equation}
\noindent
Solving for $n_{1}$ in the equation coming from the vanishing of the bracket gives $n_{1} = -2n_{2}-1$, which yields 
\begin{equation}
T_{1} = \sum_{n_{2}=0}^{\infty} \frac{(-1)^{n_{2}}}{n_{2}!}
\frac{a^{-2n_{2}-1} b^{2n_{2}}}{2^{2n_{2}}} \frac{\Gamma(2n_{2}+1)}{\Gamma(n_{2}+1)}.
\end{equation}
\noindent
To simplify this sum transform the gamma factors via \eqref{gamma-poch} 
and use the duplication formula \eqref{poch-dupl} to produce 
\begin{equation}
T_{1} = \frac{1}{a} \sum_{n_{2}=0}^{\infty} \frac{ \left(\tfrac{1}{2} \right)_{n_{2}}}{n_{2}!} 
\left( - \frac{b^{2}}{a^{2}} \right)^{n} = 
\frac{1}{a} \pFq10{\frac{1}{2}}{-}{- \frac{b^{2}}{a^{2}}}.
\end{equation}
\noindent
The identity $\displaystyle{\pFq10{c}{-}{z}  = (1-z)^{-c}}$ gives 
$\displaystyle{T_{1} = \frac{1}{\sqrt{a^{2}+b^{2}}}.}$ A direct calculation shows that the series obtained from 
solving for $n_{2}$ yields the same solution,  so it discarded. Therefore 
\begin{equation}
\int_{0}^{\infty} e^{-ax} J_{0}(bx) \, dx = \frac{1}{\sqrt{a^{2}+b^{2}}}.
\label{bessel-j0}
\end{equation}

\smallskip

The  evaluation of  this integral using Theorem \ref{nice-1} begins with checking that the boundary terms 
vanish. This comes from the asymptotic behavior $J_{0}(x) \sim 1$ as $x \to 0$ and 
$\displaystyle{J_{0}(x) \sim \sqrt{\frac{2}{\pi x}} \cos x }$ as $x \to \infty$. The term 
\begin{equation}
a \frac{\partial I(a,b)}{\partial a} = \sum_{n_{1},n_{2}} 
\phi_{n_{1}n_{2}} \frac{n_{1} a^{n_{1}} b^{2n_{2}} }{2^{2n_{2}} \Gamma(n_{2}+1)} 
\langle n_{1} + 2n_{2} + 1 \rangle.
\end{equation}
\noindent
This generates two series 
\begin{equation}
T_{1} = \frac{1}{b} \sum_{n=0}^{\infty} \frac{(-1)^{n}}{n!}  \frac{n \Gamma\left( \tfrac{1+n}{2} \right)}{\Gamma\left( 
\tfrac{1-n}{2} \right)} \left( \frac{2a}{b} \right)^{n}
\end{equation}
\noindent
and 
\begin{equation}
T_{2} = - \frac{1}{a} \sum_{n=0}^{\infty} \frac{(-1)^{n}}{n!}  \frac{\Gamma(2n+2)}{\Gamma(n+1)} 
\left( \frac{b^{2}}{4a^{2}} \right)^{n}.
\end{equation}

Similarly
\begin{equation}
b \frac{\partial I(a,b)}{\partial b} = 2 \sum_{n_{1},n_{2}} \phi_{n_{1},n_{2}} 
\frac{n_{2} a^{n_{1}} b^{2n_{2}}}{2^{2n_{2}} \Gamma(n_{2}+1)} \langle n_{1} + 2n_{2} + 1 \rangle
\end{equation}
\noindent
which yields the two series 
\begin{eqnarray*}
\tilde{T}_{1} & = & \frac{2}{b} \sum_{n=0}^{\infty}  
\phi_{n} \frac{\Gamma \left( \frac{n+1}{2} \right) \Gamma \left( - \frac{n+1}{2} \right) }
{ \Gamma \left( \frac{1-n}{2} \right)}  \left( \frac{2a}{b} \right)^{n} \\
\tilde{T}_{2} & = & \frac{2}{a} \sum_{n=0}^{\infty} \frac{2}{a} \sum_{n=0}^{\infty} \phi_{n} 
\frac{n \Gamma(2n+1)}{\Gamma(n+1)} \left( \frac{b^{2}}{4a^{2}} \right)^{n}
\end{eqnarray*}
\noindent
Since the boundary terms vanish, the relation \eqref{formula-ode1} gives 
\begin{equation}
I(a,b) = \begin{cases} 
-T_{1} - \tilde{T}_{1}, & \quad |4a^{2}| < |b^{2}| \\
-T_{2} - \tilde{T}_{2}, & \quad |b^{2}| < |4a^{2}|.
\end{cases}
\end{equation}

The form $T_{2}+\tilde{T}_{2}$ is simplified by converting them to hypergeometric form to produce 
\begin{eqnarray}
T_{2} & = & - \frac{1}{a} \sum_{n=0}^{\infty} 
\frac{(-1)^{n}}{n!} \frac{\Gamma(2n+2)}{\Gamma(n+1)} \left( \frac{b^{2}}{4a^{2}} \right)^{n} 
= - \frac{a^{2}}{(a^{2}+b^{2})^{3/2}} \\
\tilde{T}_{2} & = & \frac{2}{a} \sum_{n=0}^{\infty} \phi_{n} 
\frac{n \Gamma(2n+1)}{\Gamma(n+1)} \left( \frac{b^{2}}{4a^{2}} \right)^{n} = 
- \frac{b^{2}}{(a^{2}+b^{2})^{3/2}}.  \nonumber 
\end{eqnarray}
\noindent
Then 
\begin{equation}
I(a,b) = - T_{2} - \tilde{T}_{2} = \frac{a^{2}}{(a^{2}+b^{2})^{3/2}} + \frac{b^{2}}{(a^{2}+b^{2})^{3/2}} = 
\frac{1}{\sqrt{a^{2}+b^{2}}}.
\end{equation}
\noindent
This gives
\begin{equation}
I(a,b) = \int_{0}^{\infty} e^{-ax} J_{0}(bx) \, dx = \frac{1}{\sqrt{a^{2}+b^{2}}}.
\end{equation}
\noindent
The option $T_{1}+ \tilde{T}_{1}$ gives the same result.
\end{example}

\begin{example}
\label{example-6-222}
Entry $6.222$ in \cite{gradshteyn-2015a} is 
\begin{eqnarray}
\label{ei-double}
I(a_{1},a_{2}) & = &  \int_{0}^{\infty} \text{Ei}(-a_{1}x) \text{Ei}(-a_{2}x) \, dx  \\
& = & \left( \frac{1}{a_{1}} + \frac{1}{a_{2}} \right) \ln(a_{1}+a_{2})  - \frac{\ln a_{1}}{a_{2}} - \frac{\ln a_{2}}{a_{1}}.
\nonumber
\end{eqnarray}
\noindent
In particular 
\begin{equation}
\int_{0}^{\infty} \text{Ei}^{2}(-ax) \, dx =  \frac{2 \ln 2 }{a}.
\end{equation}

The evaluation of this integral by the method of brackets  begins with the partially divergent series 
for $\text{Ei}(-x)$  which  yields (using \eqref{ei-null1} $ = $ \eqref{pds-ei1}):
\begin{equation}
\label{sum-1a}
I(a_{1},a_{2}) = \sum_{n_{1},n_{2}} \phi_{n_{1},n_{2}} \frac{a_{1}^{n_{1}} a_{2}^{n_{2}}}{n_{1}n_{2}} 
\langle n_{1}+n_{2} + 1\rangle.
\end{equation}
\noindent
The usual procedure requires the relation $n_{1}+n_{2}+1 = 0$ and taking $n_{1}$ as the free parameter gives 
\begin{equation}
I_{1}(a_{1},a_{2}) = - \frac{1}{a_{2}} \sum_{n_{1}=0}^{\infty} \frac{(-1)^{n_{1}}}{n_{1}(n_{1}+1)} 
\left( \frac{a_{1}}{a_{2}} \right)^{n_{1}},
\end{equation}
\noindent
and when $n_{2}$ as free parameter one obtains the series 
\begin{equation}
I_{2}(a_{1},a_{2}) = - \frac{1}{a_{1}} \sum_{n_{2}=0}^{\infty} \frac{(-1)^{n_{2}}}{n_{2}(n_{2}+1)} 
\left( \frac{a_{2}}{a_{1}} \right)^{n_{2}}.
\end{equation}
\noindent
These two series correspond to different expansions: the first one in $x = a_{1}/a_{2}$ and the second one in 
$x^{-1}= a_{2}/a_{1}$.  Both series are partially divergent, so the Rule $E_{3}$ states that these sums must be 
discarded. The usual method of brackets fails for this problem.

The solution using Theorem \ref{nice-1} is described next. An elementary argument shows that 
$x {\rm{Ei}}(-x) \to 0 $ as $x \to 0$ or $\infty$. Then \eqref{formula-ode1} becomes 
\begin{eqnarray}
& & \\
I(a_{1},a_{2}) & = &  - a_{1} \frac{\partial I}{\partial a_{1}} - a_{2} \frac{\partial I}{\partial a_{2}}  \nonumber \\
& = & - \sum_{n_{1},n_{2}} \phi_{n_{1},n_{2}} \frac{a_{1}^{n_{1}} a_{2}^{n_{2}}}{n_{2}} 
\langle n_{1} + n_{2} + 1 \rangle - 
\sum_{n_{1},n_{2}} \phi_{n_{1},n_{2}} \frac{a_{1}^{n_{1}} a_{2}^{n_{2}}}{n_{1}} 
\langle n_{1} + n_{2} + 1 \rangle, \nonumber  \\
& \equiv & S_{1} + S_{2}, \nonumber
\end{eqnarray} 
\noindent
using \eqref{sum-1a} to compute the partial derivatives.  The method of brackets gives two series for each of the 
sums $S_{1}$ and $S_{2}$:
\begin{eqnarray}
T_{1,1} & = &  \frac{1}{a_{2}} \sum_{n=0}^{\infty} \frac{(-1)^{n}}{n+1} \left( \frac{a_{1}}{a_{2}} \right)^{n}  \\
T_{1,2} & = & - \frac{1}{a_{1}} \sum_{n=0}^{\infty} \frac{(-1)^{n}}{n} \left( \frac{a_{2}}{a_{1}} \right)^{n} \\
T_{2,1} & = & - \frac{1}{a_{2}} \sum_{n=0}^{\infty} \frac{(-1)^{n}}{n} \left( \frac{a_{1}}{a_{2}} \right)^{n} \\
T_{2,2} & = &  \frac{1}{a_{1}} \sum_{n=0}^{\infty} \frac{(-1)^{n}}{n+1} \left( \frac{a_{2}}{a_{1}} \right)^{n}, 
\end{eqnarray}
the series $T_{1,1}$ and $T_{1,2}$ come from the first sum $S_{1}$ and $T_{2,1}, \, T_{2,2}$ from $S_{2}$.  Rule 
$E_{3}$ indicates that the value of the integral is either 
\begin{equation}
\label{mess-1}
I(a_{1},a_{2}) = T_{1,1} + T_{2,1} \quad \text{ or } \quad I(a_{1},a_{2}) = T_{1,2} + T_{2,2};
\end{equation}
the first form is an expression in $a_{1}/a_{2}$ and the second one in $a_{2}/a_{1}$.  

The series $T_{1,1}$ is convergent when $|a_{1} | < |a_{2}|$ and it produces the function 
\begin{equation}
f(a_{1},a_{2}) = \frac{1}{a_{1}} \log \left( 1 + \frac{a_{1}}{a_{2}} \right)
\end{equation}
\noindent
and $T_{2,2}$ is also convergent and is gives 
\begin{equation}
g(a_{1},a_{2}) = \frac{1}{a_{2}} \log \left( 1 + \frac{a_{2}}{a_{1}} \right).
\end{equation}
\noindent
Observe that, according to \eqref{mess-1} to complete the evaluation of $I(a_{1},a_{2})$,  some of the series required are  
partially divergent series. The question is how to make sense of these divergent series. The solution  proposed here is, for instance, 
 to interpret 
$T_{2,1}$ as a partially divergent series 
attached to the function $g(a_{1},a_{2})$. Therefore, the sum in \eqref{mess-1}, the term $T_{2,1}$ is 
replaced by $g(a_{1},a_{2})$ to produce 
\begin{eqnarray} 
I(a_{1},a_{2}) & = & f(a_{1},a_{2}) + g(a_{1},a_{2}) \\
 & = & \frac{1}{a_{1}} \log \left( 1 + \frac{a_{1}}{a_{2}} \right) +  \frac{1}{a_{2}} \log \left( 1 + \frac{a_{2}}{a_{1}} \right),
 \nonumber
 \end{eqnarray}
 and this confirms \eqref{ei-double}.  A similar interpretation of $T_{1,2} + T_{2,2}$ gives the same result.
\end{example}

 \section{Conclusions}
 \label{sec-conclusions}
 
 The method of brackets consists of a small number of heuristic rules used for the evaluation of definite integrals on 
 $[0, \, + \infty)$. The original formulation of the method applied to functions that admit an expansion of the form 
 $\begin{displaystyle} \sum_{n=0}^{\infty} a(n) x^{\alpha n + \beta - 1} \end{displaystyle}$.  The results 
 presented here extend this method to 
 functions, like the Bessel function $K_{\nu}$ and the exponential integral $\text{Ei}$, where the expansions have expansions of the 
 form $\begin{displaystyle} \sum_{n=0}^{\infty} \Gamma(-n) x^{n} \end{displaystyle}$ (where all the coefficients are divergent) 
 or $\begin{displaystyle} \sum_{n=0}^{\infty} \frac{1}{\Gamma(-n)} x^{n} \end{displaystyle}$ (where all the coefficients vanish). 
 A variety of examples illustrate the validity of this formal procedure.

\medskip

\noindent
\textbf{Acknowledgments.} The authors wish to thank a referee for a careful reading of the 
original version of the paper. The first author thanks the support of the Centro de Astrof\'{i}sica de 
Valparaiso. The last author acknowledges the partial support of NSF-DMS 1112656. 

\bigskip

%\bibliography{/Users/vmh/Desktop/Academic/AllRef/official.bib}
%\bibliographystyle{plain}
%\end{document}

\end{document}